\documentclass[a4paper,twoside,11pt]{article}

\input{preamble}
\input{macros}

\begin{document}

\title{Dimension theory of group actions by circle diffeomorphisms I: Stationary measures}

\author{Weikun He, Yuxiang Jiao and Disheng Xu}

\date{}

\maketitle

\begin{abstract}
We establish a structure theorem for smooth random walks on the circle and study the dimensional properties of their stationary measures, including exact dimensionality and a Ledrappier-Young-type dimension formula. 
These results provide the technical tools needed for the main theorems in Part II \cite{HJX2}. 
They also yield an approximation theorem in the spirit of Katok, Avila-Crovisier-Wilkinson, and Morris-Shmerkin, together with a time-reversal symmetry for circle actions: every semigroup $T$ contained in $\Diff_+^{1+\alpha}(\SS^1)$ without finite orbits has the same number of minimal sets as its inverse semigroup $T^{-1}$.

\end{abstract}

\tableofcontents

\section{Introduction}

Group and semigroup actions on manifolds are central objects in modern dynamical systems, and it is natural to study the random walks associated with such actions. 
The two viewpoints are closely related: minimal invariant sets describe the asymptotic geometry of the deterministic action, while ergodic stationary measures encode the asymptotic statistics of the random walk. 
In this two-part series, we study the dimensions of dynamically defined sets and measures arising from smooth circle actions. 
The present paper concerns stationary measures of smooth random walks; the companion paper concerns invariant sets of group actions.

Let $\Diff_+^{1+\alpha}(\SS^1)$ be the group of orientation-preserving circle diffeomorphisms whose derivatives are $\alpha$-H\"older continuous for some $\alpha>0$. 
We will consider random walks on the circle generated by elements of $\Diff_+^{1+\alpha}(\SS^1)$; that is, let $\mu$ be a finitely supported probability measure on $\Diff_+^{1+\alpha}(\SS^1)$. 
A \emph{random walk} of law $\mu$ is defined by the random compositions $f^n_\omega \defeq f_{n-1}(\omega) \circ \dots \circ f_0(\omega)$ acting on $\SS^1$, where $(f_n(\omega))_{n\geqslant 0}$ is a sequence of independent, identically distributed random variables taking values in $\Diff_+^{1+\alpha}(\SS^1)$ with distribution $\mu$.

A Borel probability measure $\nu$ on $\SS^1$ is said to be \emph{$\mu$-stationary} if
\[\nu=\mu*\nu \defeq \int f_*\nu\dd\mu(f),\]
where $f_*\nu$ is the pushforward of $\nu$ by the diffeomorphism $f$. 
We call a $\mu$-stationary measure \emph{ergodic} if it cannot be written as a nontrivial convex combination of two distinct $\mu$-stationary measures.
The \emph{Lyapunov exponent} of an ergodic $\mu$-stationary measure $\nu$ is given by
\begin{equation}\label{eqn: Lyapunov exponent}
	\lambda(\mu,\nu)=\iint \log f'(x) \dd \mu(f)\dd\nu(x).
\end{equation}

Assume that $\supp\mu$ does not preserve any common probability measure on the circle. 
It is known that there are only finitely many ergodic stationary measures associated with such a random walk. 
Moreover, typical trajectories of this random walk are asymptotically contracting, meaning that $\lambda(\mu,\nu)<0$ for all ergodic stationary measures.

Our first main result concerns the dimensional properties of these stationary measures, in particular, their exact dimensionality and related dimension formulas.

\paragraph{Exact dimensionality of ergodic stationary measures.}
Let $\eta$ be a Borel probability measure on some metric space. 
Then $\eta$ is said to be \emph{exact dimensional} if there exists some  $\beta\geq 0$ such that  
\begin{equation}\label{eqn: exact dim}
	\lim_{\rho\to 0^+}\frac{\log\eta(B(x,\rho))}{\log\rho}=\beta,\quad\ae[\eta\mr{-}] x.
\end{equation}
In this case, we call $\beta$ the \emph{exact dimension} of $\eta$ and denote it by $\dim\eta.$
The concept of exact dimensionality usually reveals the regularity of a fractal measure and is one of its fundamental properties.   

Many fractal measures arise from dynamical iterations, for example self-similar measures of iterated function systems and stationary measures of random walks.
In many settings, these dynamically-defined measures are expected to be exact dimensional. Indeed, exact dimensionality has been studied extensively and established across various mathematical contexts, spanning from smooth ergodic theory to iterated function systems and random walks \cite{Led83, LY2, Y82, FLR, BaPeSc, FengHu, HS17, Rap, LedLe, LedLe2}.

To state our theorem on the exact dimensionality of stationary measures, we recall the notion of Furstenberg entropy for random walks.
\begin{definition}\label{def: Furstenberg entropy}
Let $\nu$ be a $\mu$-stationary measure. The \emph{Furstenberg entropy} of $(\mu,\nu)$ is defined by
\[h_{\mr F}(\mu,\nu)\defeq \iint \log \frac{\dr f_*\nu}{\dr\nu}(x)\,\dr f_*\nu(x)\dr \mu(f).\]
 \end{definition}
In this paper, we prove that stationary measures are exact dimensional for general $C^{1+\alpha}$ random walks on the circle, thereby generalizing the results in \cite{HS17,Led83} to the smooth setting.
\begin{theorem}\label{thm: exact dimensionality}
	Let $\mu$ be a finitely supported probability measure on $\Diff_+^{1+\alpha}(\SS^1)$ such that $\supp\mu$ does not preserve any probability measure on $\SS^1.$ 
	If $\nu$ is an ergodic $\mu$-stationary measure on $\SS^1,$ then $\nu$ is exact dimensional and
	\[\dim\nu=\frac{h_{\mr F}(\mu,\nu)}{|\lambda(\mu,\nu)|}. \] 
\end{theorem}
While the assumption that the support of $\mu$ does not preserve any invariant measure may appear restrictive, it is actually quite mild. 
If it were not satisfied, then either the support of $\mu$ would preserve a finite orbit, or the action of the generated group would be semi-conjugate to an action by rotations.
In the latter case, the associated stationary measure could be non-exact dimensional if its rotation number is Liouville, as shown in \cite{Sad}.

\paragraph{Dimension formulas involving the random walk entropy.}

In general, computing the Furstenberg entropy mentioned in Theorem \ref{thm: exact dimensionality} is challenging. 
In practice, alternative quantities are often needed to determine the exact dimension of a stationary measure. Under certain discreteness or separation assumptions, the random walk entropy provides a viable replacement for $h_{\mr F}$.

\begin{definition}\label{def: random walk entropy}
For a finitely supported measure $\mu$ on a group $G$, the \emph{Shannon entropy} $H(\mu)$ of $\mu$ is given by $-\sum_{f\in\supp\mu} \mu(f)\log \mu(f)$. 
Let $\mu^{*n}$ denote the $n$-fold convolution of $\mu$ in $G$. 
The \emph{random walk entropy} of $\mu$ is then defined as
\[h_{\rm{RW}}(\mu)\defeq \lim_{n\to \infty} \frac{1}{n}H(\mu^{\ast n}).\]
\end{definition}
The limit is guaranteed to exist by sub-additivity, as shown in \cite{KV83}.
The random walk entropy captures the extent to which the semigroup generated by $\supp\mu$ is free. 
Specifically, we have $\hRW(\mu) \leq H(\mu)$, and the two are equal if and only if the semigroup generated by $\supp\mu$ is freely generated by $\supp\mu$. 

In general, $\hRW$ provides an upper bound for $\hFur(\mu,\nu)$, as shown, for instance, in \cite{HS17}. 
Establishing equality between the Furstenberg entropy and the random walk entropy yields a more explicit and computable dimension formula, which greatly facilitates the dimensional analysis of stationary measures. 
This equality is typically expected when the random walk is supported on a group satisfying certain discreteness conditions. 
For example, these quantities coincide for random walks on the full flag manifold of $\RR^n$ driven by discrete subgroups of $\SL_n(\RR)$, a result deeply connected to the theory of Poisson boundaries \cite{Ka85,KV83,Le85,Fur02}. 
A breakthrough regarding dimension formulas for stationary measures was made by Hochman \cite{Hoch14}, who relaxed the discreteness requirement to an exponential separation condition for self-affine measures on the real line. 
Subsequently, Hochman-Solomyak \cite{HS17} established a similar dimension formula for stationary measures on $\RR\PP^1$ generated by $\SL_2(\RR)$ random walks satisfying exponential separation.

For general smooth random walks, a natural discreteness condition is $C^1$-local discreteness.
Recall that if $G\sbs\Homeo(\SS^1)$ has no finite orbits on $\SS^1$ then it preserves a unique $G$-minimal subset $\Lambda$ \cite[Theorem 2.1.1]{Na11}.
\begin{definition}\label{def: discreteness}
Let $G\subset\Diff_+^1(\SS^1)$ be a subgroup without finite orbits, and let $\Lambda$ be its unique minimal set.
We say that $G$ is \emph{$C^1$-locally discrete} if, for every open interval $I\subset\SS^1$ intersecting $\Lambda$, no sequence of pairwise distinct elements of $G$ converges to the identity in the $C^1$ topology on $I$. 
\end{definition}
If $\mu$ is a finitely supported probability measure on a locally discrete group $G\sbs\Diff_+^{1+\alpha}(\SS^1)$ preserving no probability measure on $\SS^1$ and its support generates $G$ as a semigroup, then it admits a unique stationary measure $\nu$. 
Deroin \cite[Theorem 1.1]{Deroin} showed that a finite $G$-equivariant quotient of $(\SS^1,\nu)$ realizes the Poisson boundary of $(G,\mu)$.
In particular, $h_{\mr F}(\mu,\nu)=\hRW(\mu)$.
In this paper, we establish the following explicit dimension formula for general random walks on the circle. 
\begin{theorem}\label{thm: dim formula C1+}
	 Let $\mu$ be a finitely supported probability measure on $\Diff^{1+\alpha}_+(\SS^1).$ 
	 Assume that the group generated by $\supp\mu$ preserves no probability measure on $\SS^1$ and is locally discrete.
	 Then $h_{\mr F}(\mu,\nu)=h_{\mr{RW}}(\mu)$ for every $\mu$-stationary measure $\nu.$ 
	 Consequently, every ergodic $\mu$-stationary measure is of exact dimension $\frac{h_{\mr{RW}}(\mu)}{|\lambda(\mu,\nu)|}.$
\end{theorem}
In fact, Theorem \ref{thm: dim formula C1+} follows from a stronger but slightly more technical version of dimension formula.
This version also plays an important role in Part II \cite{HJX2}.
Here, $T_\mu$ denotes the semigroup generated by $\supp\mu$.

\begin{theorem}\label{thm: stationary measure dimension}
    Let $\mu$ be a finitely supported probability measure on $\Diff^{1+\alpha}_+(\SS^1).$
	Let $I$ be either $\SS^1$ or a closed subinterval of $\SS^1$ which is preserved by the action of $T_\mu.$
	Assume that there is no common $T_\mu$ invariant probability measure on $I.$
	Let $\nu$ be an ergodic $\mu$-stationary measure on $I.$
	Then the Lyapunov exponent $\lambda(\mu,\nu)$ is negative, and one of the following holds:
	\begin{enumerate}[\indent(1)]
		\item\label{alt: dimension formula} either $\dim\nu=\frac{h_{\mr{RW}}(\mu)}{|\lambda(\mu,\nu)|},$
		\item\label{alt: locally nondiscrete} or there exists a closed subinterval $J\sbs I$ with $J^\circ\cap\supp\nu\ne\vn$ and two sequences $(g_n)_{n\geqslant 1},(f_n)_{n\geqslant 1}\sbs T_\mu$ with $g_n\ne f_n$ for all $n$, such that $f_n(J) \sbs g_n(I)$, and the maps $(g_n^{-1}f_n)|_J$ converge to the identity map on $J$ in the $C^1$-topology.
	\end{enumerate}
\end{theorem}
\begin{remark}
    We say that $T_\mu$ preserves $I$ if $f(I)\subset I$ for every $f\in T_\mu$.
\end{remark}
\begin{remark}
    By embedding an interval into the circle, one can also deduce a dimension formula for contracting IFSs on this interval by this theorem. 
\end{remark}

It is worth mentioning that a dimension formula of measures as in Theorem \ref{thm: dim formula C1+} usually plays an important role for the study of the associated fractal sets. 
In our companion paper, in order to understand the dimension of a fractal set preserved by a group action on the circle, this dimension formula helps us to generate a stationary measure on this set with a large enough dimension, which in turn provides an estimate of the dimension of the fractal set.

\paragraph{Approximation by uniformly hyperbolic subsystems.}
In the study of dynamical systems, particularly those exhibiting weak hyperbolicity, a natural question is to determine the size and richness of their uniformly hyperbolic subsystems.

A classical theorem of Katok \cite{Ka} asserts that any ergodic hyperbolic measure of a $C^{1+}$ diffeomorphism can be approximated by a horseshoe with approximately the same entropy. 
In \cite[Theorem 3.3]{ACW}, Avila, Crovisier and Wilkinson further showed that the horseshoe can be constructed to ensure a dominated splitting while retaining approximately the same Lyapunov exponents.
For two-dimensional matrix cocycles over a full shift of finite type, Morris and Shmerkin \cite{MS} showed that any cocycle with distinct Lyapunov exponents can be approximated by a subsystem exhibiting a dominated splitting and sharing almost the same entropy and Lyapunov exponents. 
Note that a dominated splitting can be translated into the cone field conditions for dynamics on projective spaces \cite{ABY}.

For smooth random walks on the circle, we could also consider a parallel problem. 
Our goal is to approximate the original random walk by a uniformly hyperbolic subsystem, in such a way that the dimension of the subsystem's attractor is arbitrarily close to the dimension of a given stationary measure. 
Remarkably, we can construct uniformly hyperbolic subsystems that simultaneously approximate three fundamental dynamical quantities of the random walk: the entropy, the Lyapunov exponent, and the dimension of the stationary measure. 
We present two variations of this result.

For a positive integer $n$ and a subset $\cS$ of a (semi)group, we denote $\cS^{* n}=\lb{g_1\cdots g_n:g_i\in\cS}.$
\begin{theorem}\label{thm: ACW approximation}
	Let $\cS\sbs \Diff_+^{1+\alpha}(\SS^1)$ be a finite set without common invariant probability measures. 
	Then there exist positive integers $d,r>0$ such that the following hold. 
	For every nondegenerate probability measure $\mu$ supported on $\cS$ and $\ve>0$, there exist arbitrarily large integers $N$ and a subset $\Gamma\subset\cS^{*N}$ such that:
\begin{enumerate}
\item There are exactly $d$ ergodic $\mu$-stationary measures $\nu_i$. 
The Lyapunov exponent $\lambda_i$ of each $\nu_i$ is negative.
\item  $\# \{\, (f_1,\cdots,f_N) \in \cS^N : f_N\cdots f_1 \in \Gamma\,\}\geq 2^{N(H(\mu)-\ve)}$.
\item There exists $dr$ disjoint open intervals $\lb{U_{i,j}}_{0\leq i\leq d-1,\\ 0\leq j\leq r-1}$ such that for every $(i,j)$, for all $f\in\Gamma$, we have $f\overline{U_{i,j}}\subset U_{i,j}$ and $f'(x) \in [2^{N(\lambda_i-\ve)},2^{N(\lambda_i+\ve)}]$ for all $x \in U_{i,j}$, 
\item The semigroup generated by $\Gamma$ has a unique minimal set $K_{i,j}$ in each $U_{i,j}$ with  $\dimH K_{i,j}\geq \dim\nu_i-\ve.$
\end{enumerate}
\end{theorem}
It is worth noting that the numbers $d$ and $r$ depend solely on the set $\cS$, independent of the choice of $\mu$. 
These quantities characterize the random walk and will be discussed in detail later.
When $\mu$ is supported on a locally discrete subgroup, we obtain the following variant, which provides a direct estimate on the number of the chosen group elements instead of the number of corresponding word representations. 

\begin{theorem}\label{thm: separating approximation}
	Let $\mu$ be a finitely supported probability measure in $\Diff_+^{1+\alpha}(\SS^1)$ without common invariant probability measures on $\SS^1.$ 
	Let $\cS=\supp\mu$ and assume that the group generated by $\cS$ is locally discrete. 
	Let $\nu$ be an ergodic $\mu$-stationary measure and $\lambda=\lambda(\mu,\nu)<0.$ 
	Then for every sufficiently small $\ve>0$, there exist arbitrarily large integers $N$ and a subset $\Gamma\sbs\cS^{*N}$ such that:
	\begin{enumerate}
		\item The cardinality of $\Gamma$ is at least $2^{N(h_{\mr{RW}}(\mu)-\ve)}.$
		\item There exists an open interval $U\sbs\SS^1$ which is strictly preserved by every $f\in\Gamma.$
		\item The closures of $f(U)$ are pairwise disjoint for $f\in\Gamma.$
		\item For every $f\in\Gamma,$ $x\in U,$ $f'(x)\in [2^{N(\lambda-\ve)},2^{N(\lambda+\ve)}].$
		\item The semigroup generated by $\Gamma$ has a unique minimal set $K\sbs U$ with Hausdorff dimension at least $(\dim\nu-\ve).$
	\end{enumerate}
\end{theorem}

\begin{remark}
	The discreteness condition is essential since conditions (1), (3) and (4) force $h_{\mr{RW}}(\mu)\leqslant-\lambda$. Furthermore, $h_{\mr{RW}}(\mu)$ cannot be replaced by $H(\mu)$, since  $\#\cS^{*N}$ can be substantially smaller than $2^{N H(\mu)}$ due to non-trivial relations in the group.
\end{remark}

\paragraph{A key ingredient to the proof: structure and symmetries of  random walks on the circle.}
A key ingredient in the proof is a thorough study of \textit{structures and symmetries} of smooth random walks on the circle.
A classical result of Furstenberg~\cite{Fur63} shows that, under mild assumptions, a typical random product of matrices in $\SL(2,\RR)$ acting on $\RR\PP^1$ exhibits (non-stationary) \emph{North-South dynamics}. 
We generalize this phenomenon to the smooth setting and establish a comprehensive structure theorem for $C^{1+\alpha}$ random walks on the circle (Theorems \ref{thm: structure of random walk 1} and \ref{thm: structure of random walk 2}). 
While the precise statements are somewhat lengthy and technical, roughly speaking, the typical trajectory of a random iteration of circle diffeomorphisms consists of a finite number of North-South dynamics along with certain \emph{rotational symmetries}. 
This structural description provides a clear global dynamics for smooth random walks and bridges the gap between linear $\SL(2,\RR)$ actions on $\RR\PP^1$ and non-linear dynamics.

Although we defer the precise formulations, several interesting consequences of these structure theorems warrant mention. 
The number of the copies and the rotation symmetry are entirely captured by integers $d,r>0$ appearing in Theorems \ref{thm: structure of random walk 1} and \ref{thm: structure of random walk 2}. Remarkably, these constants are not only invariant for equivalence between random walks but only depend on the dynamics of the semigroup generated by the support of $\mu$, see Theorem \ref{thm: d,r top inv}.
Moreover, these constants remain invariant when the random walk is replaced by its time-reversal.
Recall that a semigroup $T\sbs\Homeo(\SS^1)$ acting on the circle without finite orbits has only finitely many minimal invariant sets \cite{Mal17}.
As a topological application, our structure theorem directly implies the following time-reversal symmetry of semigroup actions on the circle. 
\begin{theorem}\label{thm: numb min sets}
	Let $T\subset \Diff^{1+\alpha}_+(\SS^1)$ be a semigroup without finite orbits. Then $T^{-1}\defeq\lb{f^{-1}:f\in T}$  has the same number of minimal sets as $T$.
\end{theorem}

A further application of our structure theorem is the construction of a pair of group elements that possess the strongest uniform hyperbolicity under iteration. 
In our approximation results Theorems \ref{thm: ACW approximation} and \ref{thm: separating approximation}, the subsystems exhibit uniform hyperbolicity only on restricted intervals and in one direction of composition.
To overcome this, we introduce the notion of a \emph{perfect ping-pong pair} (Definition \ref{def: perfect ping-pong pair}). 
Such a pair is uniformly hyperbolic on the entire circle and possesses very clear dynamics, allowing one to precisely predict the asymptotic past and future of every point. 
Furthermore, this pair freely generates a free subgroup of $\Homeo(\SS^1)$; this result can be viewed as a smooth $C^{1+\alpha}$ variant of Margulis's classical construction of free subgroups in $\Homeo(\SS^1)$ \cite{Mar00}.

\begin{theorem}\label{thm: ping-pong pair}
	Let $T\sbs \Diff^{1+\alpha}(\SS^1)$ be a semigroup with no invariant probability measure on $\SS^1.$ Then there exists a perfect ping-pong pair $(h_1,h_2)\sbs T.$
\end{theorem}

\paragraph{Organization.}
Section~\ref{sec: random walks} develops the structure theory of $C^{1+\alpha}$ random walks on the circle, including the associated stationary measures, the topological invariants $d$ and $r$, uniform good words, and effective convergence to the Furstenberg boundary.

Section~\ref{se:5} proves the exact dimensionality of ergodic stationary measures and identifies their dimension in terms of the Furstenberg entropy and the Lyapunov exponent.

Section~\ref{se:6} establishes the dimension formula, first for random walks on the circle and then for random walks preserving a closed interval.

Finally, Section~\ref{se:7} constructs hyperbolic elements and perfect ping-pong pairs and proves the approximation theorems stated in the Introduction.

\paragraph{Notation and conventions.}
Throughout the paper, $\SS^1=\RR/\ZZ$ is endowed with its standard orientation, metric $d$, and normalized arc length. 
Unless otherwise stated, an interval in $\SS^1$ means a connected arc. For an interval $I$, we write $|I|$ for its length and $\overline I$ for its closure, and $B(x,\rho)$ denotes the open ball of radius $\rho$ centered at $x$.
If $x,y\in\SS^1$, then $[x,y]$ denotes the closed positively oriented arc from $x$ to $y$. 
For $A\subset\SS^1$ and $\rho>0$, we write
\[
  A^{(\rho)}=\bigcup_{x\in A}B(x,\rho).
\]

For $r\geq1$, $\Diff_+^r(\SS^1)$ denotes the group of
orientation-preserving $C^r$ diffeomorphisms of $\SS^1$; the notation $\Diff_+^{1+\alpha}(\SS^1)$ has the analogous meaning, with $\alpha$-H\"older derivative. Maps act on the left, and products are composed from right to left. If $\omega=(f_n)_{n\in\ZZ}$, we set
\[
f_\omega^n=
\begin{cases}
 f_{n-1}\circ\cdots\circ f_0, & n\geq1,\\
 \Id, & n=0,\\
 f_n^{-1}\circ\cdots\circ f_{-1}^{-1}, & n\leq-1.
\end{cases}
\]
For a finite subset $\cS$ of a semigroup and $n\geq1$, we write
\[
  \cS^{*n}=\{f_n\circ\cdots\circ f_1:
  f_1,\dotsc,f_n\in\cS\}.
\]
Thus $\cS^n$ is the space of words of length $n$, whereas $\cS^{*n}$ is the set of the corresponding products. 
For a finitely supported probability measure $\mu$, the semigroup generated by $\supp\mu$ is
\[
  T_\mu=\bigcup_{n\geq1}(\supp\mu)^{*n},
\]
and $\mu^{*n}$ denotes the $n$-fold convolution of $\mu$.

For a positive integer $m$, we write $[m]=\{0,\dotsc,m-1\}$; cyclic indices are regarded as elements of $\ZZ/m\ZZ$.
The identity map is denoted by $\Id$, the fixed-point set of $f$ by $\Fix(f)$, and the pushforward of a measure $\nu$ by $f$ by $f_*\nu$. 
We use $\dim\nu$ for the exact dimension of an exact-dimensional measure, $\dimH E$ for the Hausdorff dimension of a set $E$, and
\[
  \dimH\nu=\inf\{\dimH E:\nu(E)=1\}
\]
for the Hausdorff dimension of a probability measure. The Shannon, Furstenberg, and random-walk entropies are denoted by $H$, $h_{\mr F}$, and $h_{\mr{RW}}$, respectively. 

All logarithms are to base $2$.

We use Vinogradov notations: $A\ll B$ means that $A$ is bounded above by a constant multiple of $B$ and $A\ll_\square B$ means that the implicit constant depends only on $\square.$
We write $A\asymp B$ (resp. $A\asymp_\square B$) if both $A\ll B$ and $B\ll A$ (resp. $A\ll_\square B$ and $B\ll_\square A$) hold.

\paragraph{Acknowledgement.}
W. H. is supported by the National Natural Science Foundation of China (No. 12288201). 
D. X. is supported by National Key R\&D Program of China No. 2024YFA1015100 and NSFC grant 12090010, 12090015. 
Y. J. is supported by the National Natural Science Foundation of China (No. 124B2008).
Part of this work was completed during D. X.'s visits to AMSS and the University of Chicago.
He thanks both institutions for their hospitality. 
Y. J. thanks Gang Tian for his constant support and encouragement. 

We thank Amie Wilkinson, Artur Avila, De-Jun Feng, Shaobo Gan, Fran\c cois Ledrappier, Pablo Lessa, Jialun Li, Minghui Ouyang, Wenyu Pan, Yi Shi, Wenyuan Yang for useful discussions.

\section{Random walks on \texorpdfstring{$\Diff_+^{1+\alpha}(\SS^1)$}{Diff\^{}1+(S\^{}1)}}\label{sec: random walks}
\label{se:4}

\subsection{Preliminaries on random walks}\label{sec: preliminaries on random walks}
\paragraph{Cocycles.} Let $\cS\sbs\Diff_+^1(\SS^1)$ be a finite set. 
We equip $\cS$ with the discrete topology and let $\Sigma=\cS^\ZZ$ be the product space which is compact. 
We also consider the spaces $\Sigma^+=\cS^{\ZZ_{\geqslant 0}}$ and $\Sigma^-=\cS^{\ZZ_{<0}}.$ 
Elements of $\Sigma,\Sigma^+,\Sigma^-$ are denoted by $\omega,\omega^+,\omega^-,$ respectively. 
The maps $\pi^+:\Sigma\to\Sigma^+$ and $\pi^-:\Sigma\to\Sigma^-$ refer to natural projections. 
We use $\sigma$ to denote the left shift map on $\Sigma$ or $\Sigma^+.$

For an element $\omega\in\Sigma,$ we write $\omega=(\cdots,f_{-2},f_{-1},f_0,f_1,\cdots)$ where $f_n\in\cS,n\in\ZZ.$ 
Denote
\[f_\omega^n\defeq \begin{cases}
	f_{n-1}\cdots f_1f_0& \text{if } n\geq 0;\\
	f_{n}^{-1}\cdots f_{-2}^{-1}f_{-1}^{-1}& \text{if }n<0.
\end{cases} \]
We abbreviate $f_\omega^1$ to $f_\omega.$ 
It induces an invertible cocycle over $\sigma:\Sigma\to\Sigma$ defined as
\[F:\Sigma\times\SS^1\to \Sigma\times\SS^1,\quad (\omega,x)\mapsto(\sigma\omega,f_\omega x).\]
Then the map $F$ satisfies $F^n(\omega,x)=(\sigma^n\omega,f_\omega^n x)$ for every $n\in\ZZ.$

We define $f_{\omega^+}$ and $f_{\omega^+}^n$ similarly for $\omega^+\in\Sigma^+$ and $n\geqslant 0.$ 
These maps define the forward cocycle
\[F^+:\Sigma^+\times\SS^1\to \Sigma^+\times\SS^1,\quad (\omega^+,x)\mapsto (\sigma\omega^+,f_{\omega^+}x).\]
Then $F$ is semi-conjugate to $F^+$ via the natural projection, $(\pi^+,\Id)\circ F=F^+\circ(\pi^+,\Id).$ 

\paragraph{Random walks.}
Let $\mu$ be a finitely supported probability measure on $\Diff_+^1(\SS^1)$. 
We take $\cS=\supp\mu$ which induces a cocycle $F:\Sigma\times\SS^1\to\Sigma\times\SS^1$ as above.
Equip $\Sigma$ with its Borel $\sigma$-algebra.
Let $\P=\mu^\ZZ$ be the product probability measure on $\Sigma.$ 
Similarly, the spaces $\Sigma^+$ and $\Sigma^-$ are equipped with probability measures $\P^+=\mu^{\ZZ_{\geqslant 0}}$ and $\P^-=\mu^{\ZZ_{<0}}$ respectively.
In this sense, the cocycle $F:\Sigma\times\SS^1\to\Sigma\times\SS^1$ equipped with a probability measure $\P$ on $\Sigma$ is called the \textit{random walk} induced by $\mu.$

Then the stationary measures of $\mu$ correspond to the invariant measures of the forward cocycle $F^+$ in the following way.
\begin{proposition}[{\cite[Propositions 5.5 and 5.13]{Vi14}}]
Let $\nu$ be a probability measure on $\SS^1.$ Then
    \begin{enumerate}
        \item $\nu$ is $\mu$-stationary if and only if $\P^+\times\nu$ is $F^+$-invariant.
        \item $\nu$ is an ergodic $\mu$-stationary measure if and only if $\P^+\times\nu$ is an ergodic $F^+$-invariant measure.
    \end{enumerate} 
\end{proposition}

\paragraph{Lyapunov Exponents.}
For an element $(\omega^+,x)\in\Sigma^+\times\SS^1,$ recall the \textit{Lyapunov exponent} at $(\omega^+,x)$ is defined, when the limit exists, by
\[\lambda(\omega^+, x)\defeq\lim_{n\to+\infty}\frac{1}{n}\log (f_{\omega^+}^n)'(x).\]
If $\nu$ is an ergodic $\mu$-stationary measure then $\P^+\times\nu$ is an ergodic $F^+$-invariant measure. 
By Birkhoff's ergodic theorem,
\[\lambda(\omega^+, x)=\iint\log f_{\omega^+}'(y)\dd\P^+(\omega^+)\dd\nu(y)=\iint\log f'(y)\dd\mu(f)\dd\nu(y)\]
for $\P^+\times\nu$ almost every $(\omega^+,x).$ 
This corresponds to the Lyapunov exponent $\lambda(\mu,\nu)$ given in \eqref{eqn: Lyapunov exponent}.

In general, the Lyapunov exponent can be defined similarly for every ergodic $F^+$-invariant probability measure on $\Sigma^+\times \SS^1$, or every ergodic $F$-invariant probability measure on $\Sigma\times\SS^1.$
Specifically, let $m$ be an $F$-invariant probability measure on $\Sigma\times\SS^1.$ The Lyapunov exponent of $(\omega,x)$ is given by
\[\lambda(\omega,x)\defeq\lim_{n\to+\infty}\frac{1}{n}\log (f_{\omega}^n)'(x).\]
Then for $m$-almost every $(\omega,x),$ the limit exists and coincides with the backward limit
\[\lim_{n\to-\infty}\frac{1}{n}\log (f_{\omega}^n)'(x).\]
Moreover, if $m$ is ergodic, then $\lambda(\omega,x)$ is constant almost everywhere, denoted by $\lambda(m)$. We also remark that if $m$ is an ergodic $F$-invariant probability measure and $m^+=(\pi^+,\Id)_*m$ which is an ergodic $F^+$-invariant probability measure on $\Sigma\times\SS^1,$ then $\lambda(m)=\lambda(m^+).$

\subsection{Generalities on random transformations}
\label{sec: random walks 4.1}

We denote by $P$ and $Q$ the natural projections of $\Sigma\times \SS^1$ to $\Sigma$ and $\SS^1,$ respectively. 
For a Borel probability measure $m$ on $\Sigma\times\SS^1$ with $P_*m = \P$, write
\[\dr m(\omega,x)=\dr\P(\omega)\,\dr m_\omega(x)\]
for its disintegration along $P:\Sigma \times \SS^1 \to \Sigma$ in the sense of Rokhlin.
The measure $m$ being $F$-invariant translates to the following equivariant property: for $\P$-almost every $\omega \in \Sigma$,
\begin{equation}
\label{eq:moemga equiv}
m_{\sigma(\omega)} = (f_\omega)_*m_\omega.
\end{equation}

Let $\cP = \cP_{F,\P}$ denote the set of  $F$-invariant probability measures $m$ on $\Sigma \times \SS^1$ such that $P_*m=\P$.
Let $\cP^u \sbs \cP$ (resp. $\cP^s$) denote the subset of all $u$-states (resp. $s$-states). 
Recall that $m \in \cP$ is a \emph{$u$-state} (resp. \emph{$s$-state}) if its disintegration $\omega \mapsto m_\omega$  factors through $\pi^- : \Sigma \to \Sigma^-$ (resp. through $\pi^+ : \Sigma \to \Sigma^+$).
For convenience later on, we set $\mu^+=\mu$ and define $\mu^-$ as the pushforward of $\mu$ under the map $f\mapsto f^{-1}$. 

\begin{proposition}[{\cite[Proposition 5.17]{Vi14}}]
\label{prop: stationary and state}
The map $m \mapsto Q_*m$ is a bijection between the convex set $\cP^u$ (resp. $\cP^s$) and the convex set of $\mu^+$-stationary (resp. $\mu^{-}$-stationary) measures.
\end{proposition}
Note that this bijection is clearly linear, and hence preserves ergodicity.
 
The following fact is a special case of Avila-Viana invariance principle~\cite{AV10}, which is a generalization of an earlier result in the linear case by Ledrappier~\cite{Led86}. 
\begin{theorem}[{\cite[Theorem B]{AV10}}]
\label{thm: invariance principle}
Let $m$ be an $F$-invariant probability measure on $\Sigma \times \SS^1$ with $P_*m = \P$.
If $\lambda(\omega,x) \leq 0$ holds for $m$-almost every $(\omega,x)$ then $m \in \cP^u$. 
Dually, if $\lambda(\omega,x) \geq 0$ holds for $m$-almost every $(\omega,x)$  then $m \in \cP^s$.
\end{theorem}
Consequently, every ergodic $m \in \cP$ is either a $u$-state or an $s$-state. 
It follows that $\cP$ is the convex hull of $\cP^u \cup \cP^s$.
We also have the following.
\begin{corollary}
\label{cor:u state and lambda}
Assume that $\supp\mu$ does not preserve any probability measure on $\SS^1.$ 
Then for every ergodic $u$-state $m$, $\lambda(m)<0.$ 
For every ergodic $s$-state $m,$ $\lambda(m)>0.$
\end{corollary}
\begin{proof}
The equality $\lambda(m)=0$ holds if and only if $m$ is both a $u$-state and an $s$-state. 
Then the conditional measure $m_\omega$ is constant, which is a Borel probability measure on $\SS^1$ invariant for every $f\in\supp\mu.$
\end{proof}

We will also need a result of Ruelle-Wilkinson~\cite{RW01} about invertible cocycles with negative Lyapunov exponents.
\begin{theorem}[{\cite[Theorem II]{RW01}}]
Let $m$ be an ergodic $F$-invariant probability measure on $\Sigma \times \SS^1$.
Assume that $\lambda(m)<0.$ 
Then there exists a subset $X\sbs\Sigma\times\SS^1$ and a positive integer $k$ such that
	\begin{enumerate}
		\item $m(X)=1$ and
		\item for every $(\omega,x)\in X,$ $\#(X\cap \lb{\omega}\times\SS^1)=k.$
	\end{enumerate}
\end{theorem}
By considering the inverse cocycle, we see that the same holds if we assume $\lambda(m) > 0$.
Using ergodicity and equivariance~\eqref{eq:moemga equiv}, we have the following fact about the disintegration of $m$.
\begin{corollary}
\label{cor: atomic conditional measure}
If $m \in \cP$ is ergodic and $\lambda(m) \neq 0$,
then there exists a positive integer $k$ such that for $\P$-almost every $\omega \in \Sigma$, $m_\omega$ is a uniform probability measure on a set of $k$ elements.
\end{corollary}

The next result is due to Malicet \cite{Mal17}.
Recall that $T_\mu$ is the semigroup generated by $\supp\mu.$
\begin{theorem}[{\cite[Theorem B]{Mal17}}]
\label{thm: supports of stationary measures}
Let $\mu$ be a finitely supported probability measure on $\Homeo(\SS^1)$ such that $\supp\mu$ does not preserve any Borel probability measure on $\SS^1.$
Then there are only finitely many ergodic $\mu$-stationary measures on $\SS^1.$
Their topological supports are pairwise disjoint and are exactly the $T_\mu$-minimal sets.
\end{theorem}

\subsection{The structure of random walks on \texorpdfstring{$\Diff^{1+\alpha}_+(\SS^1)$}{Diff\^{}1+(S\^{}1)}}\label{sec: random walks 4.3}
Fix $\alpha>0 .$
Let $\mu$ be a finitely supported probability measure on $\Diff_+^{1+\alpha}(\SS^1)$ without common invariant probability measures on $\SS^1.$ 
In order to study the structure of the random walk induced by $\mu$, we first establish some relations among all the ergodic $\mu^\pm$-stationary measures. 
One basic question is whether the number $k$ in Corollary \ref{cor: atomic conditional measure} is the same among different $\mu^\pm$-stationary measures.

This question can be answered by constructing a dynamically defined transitive permutation among all the ergodic $\mu^\pm$-stationary measures, see Lemma \ref{lem: R is transitive}. 
The construction is partially inspired by Hertz-Hertz-Tahzibi-Ures \cite{HHTU}, where they associate each invariant measure of positive center exponent to one of negative center exponent by considering the extremal points of the Pesin center manifold. 
Using this permutation, we can show that these stationary measures share similar properties. 
Thus the stationary measures form a coherent dynamical structure on $\SS^1.$

\begin{definition}
	Let $f\in\Diff^{1+\alpha}_+(\SS^1)$ and let $I\sbs\SS^1$ be an interval.
Define the \textit{distortion coefficient} of $f$ on $I$ by
\[\vk(f,I)\defeq\sup_{x,y\in I}\abs{\log f'(x)-\log f'(y)}.\]
We also consider the \textit{$\alpha$-H\"older distortion norm} of $f$ on $I$ given by
\[\vk_\alpha(f,I)\defeq \sup_{x\ne y\in I}\frac{\abs{\log f'(x)-\log f'(y)}}{d(x,y)^\alpha}.\]
\end{definition}
Clearly $\vk(f,I)\leqslant |I|^\alpha\cdot \vk_\alpha(f,I).$
Note that for every $f,g\in\Diff_+^{1+\alpha}(\SS^1)$ we have  
 $\vk(fg,I)\leq\vk(g,I)+\vk(f,g(I))$ and
\begin{equation}\label{eqn: vk alpha fg}
	\vk_\alpha(fg,I)\leqslant  \vk_\alpha(g,I)+\|g'\|^\alpha\cdot\vk_\alpha(f,g(I)). 
\end{equation}

We now return to the cocycle $F:\Sigma\times\SS^1\to \Sigma\times \SS^1.$
We consider the subset $\Theta^s \sbs \Sigma \times \SS^1$ given by
\[\Theta^s = \lb{\,(\omega,x) \in \Sigma \times \SS^1: \limsup_{n \to +\infty}\frac{1}{n} \log (f_\omega^n)'(x) < 0\,}.\]
It is clear that $\Theta^s$ is $F$-invariant. 
\begin{lemma}\label{lem: expdecay f'}
	For every $(\omega,x)\in\Theta^s,$ there are constants $\delta,c,C > 0$ such that 
\begin{equation}
\label{eq:expdecay f'}
(f_\omega^n)'(y) \leq C2^{-cn}, \quad\forall y\in B(x,\delta),\,\forall n\geq 0.
\end{equation} 
\end{lemma}
\begin{proof}
	Choose $\delta>0$ sufficiently small, as specified below.
	We prove by an induction on $n$ that $\vk(f_\omega^n,B(x,\delta))\leqslant 1.$
	Then the conclusion follows directly by $(f_\omega^n)'(x)\to 0$ exponentially fast.
	Assume that $\vk(f_\omega^k,B(x,\delta))\leqslant 1$ holds for all $k<n.$
	Write $\omega=(\cdots,f_0,f_1,\cdots)$ where $f_k\in\cS.$
	Then we have
	\[\vk(f_\omega^n,B(x,\delta))\leqslant \sum_{k=0}^{n-1}\vk(f_{k},f_\omega^k(B(x,\delta)))\leqslant \sum_{k=0}^{n-1}\vk_\alpha(f_{k})\cdot |f_\omega^k(B(x,\delta)) |^\alpha. \]
	By the inductive assumption, 
	\[ |f_\omega^k(B(x,\delta)) |\leqslant 2\delta\cdot (f_\omega^k)'(x)\cdot 2^{\vk( f_\omega^k,B(x,\delta))}\leqslant 4\delta\cdot (f^k_\omega)'(x). \]
	Note that $\vk_\alpha(f_k)$ is uniformly bounded and $\sum_{k=0}^\infty [(f_\omega^k)'(x)]^\alpha$ converges.
	One can choose $\delta>0$ small enough that the estimate $\vk(f_\omega^n,B(x,\delta))\leqslant 1$ holds. 
\end{proof}

Therefore, for any $\omega \in \Sigma$, the slice $W^s(\omega) = \lb{x \in \SS^1 : (\omega,x) \in \Theta^s}$ is open in $\SS^1$.
Moreover $W^s(\omega) \neq \SS^1$ since otherwise, using compactness, we could cover $\SS^1$ by finitely many open intervals with the property of~\eqref{eq:expdecay f'} leading to $(f_\omega^n)'(y) \leq C2^{-cn}$ uniformly in $y \in \SS^1$, which is absurd.

For $(\omega,x) \in \Theta^s$, let $W^s(\omega,x)$ denote the connected component of $W^s(\omega)$ containing $x$.
It is an open interval of $\SS^1$.
Using \eqref{eq:expdecay f'}, we see that
\[W^s(\omega,x) = \lb{\,y\in\SS^1: \limsup_{n \to +\infty}\frac{1}{n}\log d(f_\omega^n(y),f_\omega^n(x)) < 0\,}.\]

Let $R^s(\omega,x) = (\omega,y)$ where $y$ is the right endpoint of $W^s(\omega,x)$. 
Clearly $R^s(\omega,x) \not\in \Theta^s$.
Thus, we get a map $R^s \colon \Theta^s \to \Sigma \times \SS^1  \setminus \Theta^s$.
From the $F$-invariance of $\Theta^s$, we find that $R^s$ commutes with $F$.

Let $m \in \cP^u$. 
By Corollary~\ref{cor:u state and lambda}, we have $m(\Theta^s) = 1$.
Thus, $R^s_*m$ is a well-defined Borel probability measure on $\Sigma \times \SS^1$. 
Moreover, $R^s_*m \in \cP$ since $R^s \circ F = F \circ R^s$ and $P \circ R^s = P$.
Note that $R^s_*m(\Sigma \times \SS^1 \setminus \Theta^s) = 1$, implying that for $R^s_*m$-almost every $(\omega,x)$, $\lambda(\omega,x) \geq 0$.
By Theorem~\ref{thm: invariance principle}, $R^s_*m \in \cP^s$.
To summarize, $m \mapsto R^s_*m$ is a map $R^s_* \colon \cP^u \to \cP^s$.
Similarly, we define the left-endpoint map $L^s \colon \Theta^s \to \Sigma \times \SS^1 \setminus \Theta^s$.

In a dual manner (considering the inverse cocycle $F^{-1}$), we  define $\Theta^u$, $W^u(\omega,x)$ for $(\omega,x)\in \Theta^u$. We consider  $R^u,L^u\colon \Theta^u \to \Sigma \times \SS^1 \setminus \Theta^u$ and their induced maps $R^u_*,L^u_* \colon \cP^s \to \cP^u.$

\begin{lemma}\label{lem: R injective}
Let $m \in \cP^u$. 
For $\P$-almost every $\omega \in \Sigma$, the map $x \mapsto Q\circ R^s(\omega,x)$ is injective on $\supp m_\omega$. 
\end{lemma}
\begin{proof}
By Corollary~\ref{cor: atomic conditional measure}, for $\P$-almost every $\omega \in \Sigma$, $\supp m_\omega$ is finite. 
It follows that there exists $c>0$ such that the set
\[\Sigma'=\lb{\omega: d(x,x') \geq c ,\; \forall x \ne x' \in \supp m_\omega}\]
has a positive $\P$-measure.
By the ergodicity of $\sigma$, for $\P$-almost every $\omega \in \Sigma,$ there are infinitely many $n\in\NN$ such that $\sigma^n\omega \in \Sigma'$.
For such $\omega$, for every $x \ne x' \in \supp m_\omega$, we have $d\bigl(f_\omega^n (x),f_\omega^n (x')\bigr)\not\to 0$.
In particular, $W^s(\omega,x) \ne W^s(\omega,x')$.
\end{proof}

For $\omega \in \Sigma$, let
\[\Pi(\omega) = \bigcup_{m \in \cP^u} \supp m_\omega\quad
\text{and} \quad \Xi(\omega) = \bigcup_{m \in \cP^s} \supp m_\omega.\]
Clearly, the union can be taken over ergodic $u$-states (resp. $s$-states) and it still defines the same set. Then for $\P$-almost every $\omega \in \Sigma$, $\Pi(\omega)$ and $\Xi(\omega)$ are finite sets by Theorem \ref{thm: supports of stationary measures}.

\begin{lemma}
\label{lem: Pi and Wu}
For $\P$-almost every $\omega \in \Sigma$, we have $\#\Pi(\omega) = \#\Xi(\omega)$ and 
\[\Pi(\omega) = \SS^1 \setminus W^u(\omega) \quad \text{and} \quad \Xi(\omega) = \SS^1 \setminus W^s(\omega).\]
\end{lemma}
\begin{proof}
Applying Lemma~\ref{lem: R injective} to a suitable convex combination of ergodic $u$-states and remembering $R^s_* m \in \cP^s$, we see that for $\P$-almost every $\omega \in \Sigma$, the map $x \mapsto Q\circ R^s(\omega,x)$ is injective from $\Pi(\omega)$ to $\Xi(\omega)$.
Similarly,  $x \mapsto Q\circ L^s(\omega,x)$ is injective from $\Pi(\omega)$ to $\Xi(\omega)$ and $x \mapsto Q\circ R^u(\omega,x)$ and $x \mapsto Q\circ L^u(\omega,x)$ are injective maps from $\Xi(\omega)$ to $\Pi(\omega)$, proving the claim about cardinality.

Moreover, this shows that, $W^s(\omega,x)$, $x \in \Pi(\omega)$ are disjoint open intervals with endpoints in $\Xi(\omega)$.
The total number of endpoints being equal to the number of intervals, forces these intervals and these points to partition the circle. 
Hence $\Xi(\omega) = \SS^1 \setminus W^s(\omega)$.
\end{proof}

\begin{lemma}
\label{lem: R ergodic is ergodic}
The numbers of ergodic $u$-states and $s$-states are equal.
If $m \in \cP^u$ is ergodic then so is $R^s_*m \in \cP^s$.
\end{lemma}
\begin{proof}
Let $\{m_0,\cdots,m_{d-1}\} \sbs \cP^u$ be the set of ergodic $u$-states. 
Since they are singular to each other, for $\P$-almost every $\omega \in \Sigma$, $(\supp\, (m_i)_\omega)_{0 \leq i \leq d-1}$ form a partition of $\Pi(\omega)$.
Set $m_i'=R^s_* m_i.$ 
Then, by Lemma~\ref{lem: R injective}, $(\supp\, (m'_i)_\omega)_{0 \leq i \leq d-1}$ are pairwise disjoint.
It follows that the number of ergodic $s$-states, denoted by $d'$, is no less than $d$.
The equality $d'= d$ holds if and only if every $m_i'$ is ergodic. 
But, using $R^u$, we see that indeed $d \geq d'.$ 
Hence, $d=d'$ and the lemma is proved.
\end{proof}

\begin{lemma}
\label{lem: R is transitive}
The maps $R^s_*$ and $R^u_*$ together induce a transitive permutation of ergodic measures in $\cP$. 
In particular, there is an integer $r \geq 1$ such that for every ergodic $m \in \cP$, $\#\supp m_\omega = r$ for $\P$-almost every $\omega \in \Sigma$.
\end{lemma}
\begin{proof}
The proof of Lemma~\ref{lem: Pi and Wu} shows also the following for $\P$-almost every $\omega \in \Sigma$. 
Write $\Pi(\omega) = \{x_0,\dotsc,x_{k-1}\}$ with $x_0,\dotsc,x_{k-1}$ arranged in cyclic order and set $y_j =Q\circ R^s(\omega, x_j)$ so that $y_j \in \Xi(\omega)$ for each $j$. 
Then $x_0,y_0, \dotsc, x_{k-1}, y_{k-1}$ are arranged in cyclic order.
Moreover $R^u(\omega,y_j) = (\omega,x_{j + 1 \mod k})$. In particular, $R^s$ and $R^u$ together induce a transitive permutation of points in $\Pi(\omega)\cup\Xi(\omega).$

Then the claim of the lemma follows from Lemma~\ref{lem: R ergodic is ergodic} together with the observation that two ergodic measures $m,m'$ in $\cP$ are equal if and only if $m_\omega$ and $m'_\omega$ have common atoms for a set of $\omega$ with positive $\P$ measure. 
The ``in particular'' part follows from the fact that $R^s_*$ and $R^u_*$ preserve $\#\supp m_\omega$, a direct consequence of Lemma \ref{lem: R injective}.
\end{proof}

Let $d$ denote the number of ergodic $u$-states.
We write $\Zmodd = \{0, \dotsc, d-1\}$ and the addition in $\Zmodd$ is to be understood reduction modulo $d$. 
Fix an arbitrary ergodic $u$-state $m^+_0 \in \cP^u$.
By Lemma~\ref{lem: R is transitive}, applying alternatively $R^u_*$ and $R^s_*$, we find in the sequence
\[m_0^+\mapsto m_0^-\mapsto m_1^+\mapsto\cdots m_{d-1}^+\mapsto m_{d-1}^-\mapsto m_0^+\]
all ergodic $u$-states $m_i^+$ and all ergodic $s$-states $m_i^-$, $i \in \Zmodd$.
For $i \in \Zmodd$ and $\omega \in \Sigma$, define 
\[\Pi(\omega,i) = \supp (m^+_i)_\omega \quad \text{and} \quad \Xi(\omega,i) = \supp (m^-_i)_\omega.\]

Then for each $i\in\Zmodd$, we have that $Q\circ (R^u \circ R^s)^d$ preserves $\Pi(\omega,i)$ and maps every element of $\Pi(\omega,i)$ to the element on its right.

\paragraph{The space $\SS_k$.}
For a positive integer $k,$ let $\SS_k$ denote the family of $k$-element subsets of $\SS^1.$
For later convenience, we equip $\SS_k$ with a natural metric, making it a metric space. 
This metric is defined as
\[\ul d(\ul x,\ul y)= \inf_{\tau\in \mathfrak S_k}\sum_{i=0}^{k-1} d(x_i,y_{\tau(i)})\]
where $\ul x=\lb{x_0,\cdots,x_{k-1}},\ul y=\lb{y_0,\cdots,y_{k-1}}\in\SS_k$ and $\mathfrak S_k$ is the symmetric group of $[k].$ 
Therefore, every element in $\Homeo(\SS^1)$ can be naturally regarded as an element in $\Homeo(\SS_k).$

For every $\ul x\in\SS_k,$ we define the probability measure $u_{\ul x}$ to be the uniform measure on $\ul x.$ 
Let $\cM(\SS^1)$ be the space of Radon measures on $\SS^1$ equipped with the weak-* topology. 
Then $u_{\bullet}:\ul x\mapsto u_{\ul x}$ is a topological embedding. 
In fact, it is a homothety onto its image if we equip $\cM(\SS^1)$ with the Wasserstein metric
\[\mr{dist}(\eta,\zeta)=\sup_\phi \abs{\int\phi\dd\eta-\int\phi\dd\zeta}, \]
where the supreme is taken over all $1$-Lipschitz functions and $\eta,\zeta$ are in $\cM(\SS^1).$

\paragraph{Structure of random walks on $\Diff_+^{1+\alpha}(\SS^1)$.}

The following is immediate from our discussion.
\begin{theorem}[Structure of random walks I: construction]\label{thm: structure of random walk 1}
Assume that $\supp\mu$ does not preserve any probability measure on $\SS^1.$ 
Then there exist two positive integers $d,r$ and two measurable maps $\Pi:\Sigma\times \Zmodd \to \SS_r$ and $\Xi:\Sigma\times \Zmodd \to\SS_r$ such that
\begin{enumerate}
\item Let $m_i^\pm$ be probability measures on $\Sigma\times\SS^1$ defined by
\[\dr m_i^+=\dr\P(\omega)\dr \unif_{\Pi(\omega,i)},\quad \dr m_i^-=\dr\P(\omega)\dr \unif_{\Xi(\omega,i)},\]
then the measures $m_i^+$ (resp. the measures $m_i^-$) are exactly the ergodic $u$-states (resp. $s$-states) that projects to $\P.$
\item Let $\nu_i^\pm$ be probability measures on $\SS^1$ defined by $\nu_i^+=Q_* m_i^+$ and $\nu_i^-=Q_*m_i^-$. They can also be expressed as
\[\nu_i^+=\int \unif_{\Pi(\omega,i)}\dd\P(\omega),\quad \nu_i^-=\int \unif_{\Xi(\omega,i)}\dd\P(\omega).\]
Then $\nu_i^+$'s (resp. $\nu_i^-$'s) are exactly the ergodic $\mu^+$-stationary (resp. $\mu^-$-stationary) measures.
\end{enumerate}
Moreover, for $\P$ almost every $\omega,$ the following holds
\begin{enumerate}
\item[(3)] $\Pi(\omega,i)$ only depends on $(\pi^-\omega,i)$ and $\Xi(\omega,i)$ only depends on $(\pi^+\omega,i).$ 
\item[(4)] Cocycle invariance: $\forall n\in\ZZ,$ we have $f_\omega^n\Pi(\omega,i)=\Pi(\sigma^n\omega,i),$ $f_\omega^n\Xi(\omega,i)=\Xi(\sigma^n\omega,i).$
\item[(5)] Define the sets on $\SS^1$ as
\[\Pi(\omega)=\bigcup_{i\in[d]}\Pi(\omega,i),\quad \Xi(\omega)=\bigcup_{i\in[d]}\Xi(\omega,i).\]
Then $\Pi(\omega)\cup\Xi(\omega)$ consists of $2dr$ distinct points. Let $x_0,x_1,\cdots,x_{2dr-1}$ be these points arranged in cyclic order on $\SS^1$ such that $x_0\in \Pi(\omega,0).$ 
Then 
\[x_{2j}\in\Pi(\omega,j\mod d),\ x_{2j+1}\in\Xi(\omega,j\mod d),\quad\forall 0\leq j\leq dr-1.\]
\end{enumerate}
\end{theorem}
\begin{remark}
	Since $\Pi,\Xi$ only depends on $\pi^-\omega,\pi^+\omega,$ we will sometimes use the notation $\Pi(\omega^-)$ and $\Xi(\omega^+)$ for $\omega^-\in\Sigma^-$ and $\omega^+\in\Sigma^+,$ respectively.
\end{remark}

\begin{remark}
The map $\omega \mapsto \Pi(\omega,i)$ plays the role of the Furstenberg boundary map.
It can also be defined as follows (see \cite[Lemma 5.22]{Vi14}). 
For $\P$-almost every $\omega \in \Sigma$, we have
\begin{equation}
\label{eqn: F boundary}
(f_{\sigma^{-n}\omega}^n)_*\nu_i^+\lto{\text{weak-*}} \unif_{\Pi(\omega,i)},
\end{equation}
as $n \to +\infty$. An effective version of \eqref{eqn: F boundary} will be shown in Section \ref{sec: effective convergence}.
\end{remark}

Moreover, we can show the constants $d,r$ given in this theorem are topological invariants. 
The proof is deferred to Section \ref{sec: properties of stationary measures}, where we will show that $dr$ is indeed the least number of pairs of topologically hyperbolic fixed points in the semigroup $T_\mu,$ where we use the construction of hyperbolic elements in Section \ref{subsec: hyperbolic elements}.
\begin{theorem}\label{thm: d,r top inv}
Let $\mu$ be as in the previous theorem and $d,r$ be the constants given there.
Then $d,r$ are topological invariants of the semigroup $T_\mu$.
Specifically, if $\mu,\mu'$ are two probability measures on $\Diff^{1+\alpha}_+(\SS^1)$ satisfying the assumption of the previous theorem and that there exists $h\in\Homeo(\SS^1)$ with $T_\mu=hT_{\mu'}h^{-1},$ then the constants $d,r$ are the same for $\mu$ and $\mu'.$
\end{theorem}

The next statement supplements Theorem~\ref{thm: structure of random walk 1} with more dynamical information, that is, how $f^n_\omega$ behaves when $n$ is large.
\begin{theorem}[Structure of random walks II: dynamics]\label{thm: structure of random walk 2}
Assume that $\supp\mu$ does not preserve any probability measure on $\SS^1.$ Let $\nu_i^+$ and $\nu_i^-$, $i \in [d]$ be the ergodic stationary measures defined in the previous theorem corresponding to $\Pi(\cd,i)$ and $\Xi(\cd,i).$ 
Then for $\P$-almost every $\omega$ and $i\in[d],$ there exist subsets $W^s(\omega,i)\sbs\SS^1,W^u(\omega,i)\sbs\SS^1$ satisfying
\begin{enumerate}
	\item Each $W^s(\omega,i),W^u(\omega,i)$ is a disjoint union of $r$ open intervals.
	\item For every $i\in[d],$ $\Pi(\omega,i)\sbs W^s(\omega,i)\sbs\SS^1\sm\Xi(\omega)$ and $\Xi(\omega,i)\sbs W^u(\omega,i)\sbs\SS^1\sm\Pi(\omega).$
	\item the sets $W^s(\omega,i)$, $i \in [d]$ are pairwise disjoint and we have $\bigcup_{i\in[d]}W^s(\omega,i) =\SS^1\sm\Xi(\omega).$ The same holds for $\bigcup_{i\in[d]}W^u(\omega,i)=\SS^1\sm\Pi(\omega).$
	\item Cocycle invariance: $\forall n\in\ZZ,$ we have $f_\omega^nW^s(\omega,i)=W^s(\sigma^n\omega,i)$ and $f_\omega^nW^u(\omega,i)=W^u(\sigma^n\omega,i).$
	\item For every connected component $I$ of $W^s(\omega,i)$ and $J$ of $W^u(\omega,i)$, we have $\nu_i^+(I)=1/r$ and $\nu_i^{-}(J)=1/r.$ In particular, 
	\[\nu_i^+(W^s(\omega,i))=1,\quad \nu_i^-(W^u(\omega,i))=1,\quad\forall i\in[d]. \]
	\item For every $x\in\Pi(\omega,i)$ and $y\in\Xi(\omega,i),$
	\[\lim_{n\to\pm\infty}\frac{1}{n}\log(f^n_\omega)'(x)=\lambda_i^+,\quad \lim_{n\to\pm\infty}\frac{1}{n}\log(f^n_\omega)'(y)=\lambda_i^-.\]
	Here $\lambda_i^+<0$ and $\lambda_i^->0$ are the Lyapunov exponents corresponding to $\nu_i^+$ and $\nu_i^-.$
	\item For all closed subintervals $I\sbs W^s(\omega)$ and $ J\sbs W^u(\omega),$ as $n\to+\infty,$ we have
	\[|f_\omega^nI|\to 0,\quad |f^{-n}_\omega J|\to 0,\quad\text{exponentially fast.} \] 
\end{enumerate}
\end{theorem}
\begin{proof}
For $\omega\in\Sigma$ and $i\in[d],$ define 
\begin{equation}\label{eqn: su-manifolds}
W^s(\omega,i)=\bigcup_{x \in \Pi(\omega,i)} W^s(\omega,x)\quad \text{and}\quad W^u(\omega,i)=\bigcup_{x \in \Xi(\omega,i)} W^u(\omega,x),
\end{equation}
so that $W^s(\omega) = \bigcup_{i\in[d]}W^s(\omega,i)$ and $W^u(\omega) = \bigcup_{i\in[d]}W^u(\omega,i)$.
All items except (5) follow immediately from the discussion above.
In particular, (3) is Lemma~\ref{lem: Pi and Wu} and (7) follows from \eqref{eq:expdecay f'} and compactness.

It remains to prove (5). 
Without loss of generality, we show the argument for $\nu_i^+.$ 
Then the boundary points of $ W^s(\omega,i)$ only depends on $\pi^+(\omega)$ and hence $ W^s(\omega,i)$ only depends on $\pi^+(\omega).$ 
Fix an $\omega^+\in\Sigma^+$ and take a connected component of $I$ of $ W^s(\omega^+,i).$ 
Then for every $\omega\in\Sigma$ with $\pi^+\omega=\omega^+,$ we have $\#(I\cap \Pi(\omega,i))=1.$ 
Since $\Pi(\omega,i)$ only depends on $\pi^-\omega,$ it follows that
	\[\nu_i^+(I)=\frac{1}{r}\int_{\Sigma^-} \#\lb{I\cap \Pi(\omega^-,i)}\dr\P^-(\omega^-)=\frac{1}{r}.\qedhere \]
\end{proof}

\begin{example}
	Figure \ref{fig:417} illustrates the dynamics of a random walk in the case $d=r=2$. Here $n$ is a large positive integer and $\omega'=\sigma^n\omega$. The red points denote the points in $\Pi(\cd)$, and the blue points denote the points in $\Xi(\cd)$. The black arcs map to the black arcs with some contraction.

\begin{figure}[!ht]
    \begin{tikzpicture}
        \def\radius{2.5cm}
		\def\radone{2.9cm}
		\def\radtwo{3.3cm}
		\def\radthree{3.1cm}
		
		\draw[gray] (0,0) circle (\radius);
		\draw[-,very thick] (-35:\radius) arc[radius=\radius, start angle=-35, end angle=35];
		\draw[-,very thick] (55:\radius) arc[radius=\radius, start angle=55, end angle=125];
		\draw[-,very thick] (145:\radius) arc[radius=\radius, start angle=145, end angle=215];
		\draw[-,very thick] (235:\radius) arc[radius=\radius, start angle=235, end angle=305];
		
		\fill[red] (0,0) ++(0:\radius) circle[radius=2pt];
		\fill[red] (0,0) ++(90:\radius) circle[radius=2pt];
		\fill[red] (0,0) ++(180:\radius) circle[radius=2pt];
		\fill[red] (0,0) ++(270:\radius) circle[radius=2pt];
		
		\fill[blue] (0,0) ++(45:\radius) circle[radius=2pt];
		\fill[blue] (0,0) ++(135:\radius) circle[radius=2pt];
		\fill[blue] (0,0) ++(225:\radius) circle[radius=2pt];
		\fill[blue] (0,0) ++(315:\radius) circle[radius=2pt];
		
		\fill (0,0) ++(-35:\radius) circle[radius=1.5pt];
		\fill (0,0) ++(35:\radius) circle[radius=1.5pt];
		\fill (0,0) ++(55:\radius) circle[radius=1.5pt];
		\fill (0,0) ++(125:\radius) circle[radius=1.5pt];
		\fill (0,0) ++(145:\radius) circle[radius=1.5pt];
		\fill (0,0) ++(215:\radius) circle[radius=1.5pt];
		\fill (0,0) ++(235:\radius) circle[radius=1.5pt];
		\fill (0,0) ++(305:\radius) circle[radius=1.5pt];

		\node at (0:\radtwo) {$\Pi(\omega,0)$};
		\node at (90:\radone) {$\Pi(\omega,1)$};
		\node at (180:\radtwo) {$\Pi(\omega,0)$};
		\node at (270:\radone) {$\Pi(\omega,1)$};
		
		\node at (45:\radthree) {$\Xi(\omega,0)$};
		\node at (135:\radthree) {$\Xi(\omega,1)$};
		\node at (225:\radthree) {$\Xi(\omega,0)$};
		\node at (315:\radthree) {$\Xi(\omega,1)$};
		
		\draw[gray] (8,0) circle (\radius);
		\draw[-,very thick] (8,0)++(12.5:\radius) arc[radius=\radius, start angle=12.5, end angle=32.5];
		\draw[-,very thick] (8,0)++(102.5:\radius) arc[radius=\radius, start angle=102.5, end angle=122.5];
		\draw[-,very thick] (8,0)++(192.5:\radius) arc[radius=\radius, start angle=192.5, end angle=212.5];
		\draw[-,very thick] (8,0)++(282.5:\radius) arc[radius=\radius, start angle=282.5, end angle=302.5];
		
		\fill[red] (8,0) ++(22.5:\radius) circle[radius=2pt];
		\fill[red] (8,0) ++(112.5:\radius) circle[radius=2pt];
		\fill[red] (8,0) ++(202.5:\radius) circle[radius=2pt];
		\fill[red] (8,0) ++(292.5:\radius) circle[radius=2pt];
		
		\fill[blue] (8,0) ++(67.5:\radius) circle[radius=2pt];
		\fill[blue] (8,0) ++(157.5:\radius) circle[radius=2pt];
		\fill[blue] (8,0) ++(247.5:\radius) circle[radius=2pt];
		\fill[blue] (8,0) ++(337.5:\radius) circle[radius=2pt];
		
		\fill (8,0)++(12.5:\radius) circle[radius=1.5pt];
		\fill (8,0)++(32.5:\radius) circle[radius=1.5pt];
		\fill (8,0)++(102.5:\radius) circle[radius=1.5pt];
		\fill (8,0)++(122.5:\radius) circle[radius=1.5pt];
		\fill (8,0)++(192.5:\radius) circle[radius=1.5pt];
		\fill (8,0)++(212.5:\radius) circle[radius=1.5pt];
		\fill (8,0)++(282.5:\radius) circle[radius=1.5pt];
		\fill (8,0)++(302.5:\radius) circle[radius=1.5pt];

		\node at ($(8,0)+(22.5:\radtwo)$) {$\Pi(\omega',0)$};
		\node at ($(8,0)+(112.5:\radthree)$) {$\Pi(\omega',1)$};
		\node at ($(8,0)+(202.5:\radtwo)$) {$\Pi(\omega',0)$};
		\node at ($(8,0)+(292.5:\radthree)$) {$\Pi(\omega',1)$};
		
		\node at ($(8,0)+(67.5:\radthree)$) {$\Xi(\omega',0)$};
		\node at ($(8,0)+(157.5:\radtwo)$) {$\Xi(\omega',1)$};
		\node at ($(8,0)+(247.5:\radthree)$) {$\Xi(\omega',0)$};
		\node at ($(8,0)+(337.5:\radtwo)$) {$\Xi(\omega',1)$};
		
		\node at ($(5,2)+(260:4.5)$) {$f_\omega^n$};
		\draw[->,>=latex,very thick] (5,2)++(250:5) arc[radius=\radius, start angle=250, end angle=295];

   \end{tikzpicture}
   \stepcounter{theorem}
   \caption{A typical map $f^n_\omega : \SS^1 \to \SS^1$.}
   \label{fig:417}
\end{figure}
\end{example}

For convenience in later discussions, we define ``typical'' words in this system.

\begin{definition}\label{def: regular for random walk}
	We say $\omega\in\Sigma$ is \textit{regular for the random walk} if it
	\begin{enumerate}
		\item satisfies all the conditions in the previous two theorems, and
		\item $(\omega,x)$ is Birkhoff regular for every $m_i^+$ for every $x\in \Pi(\omega,i)$ and $(\omega,y)$ is Birkhoff regular for every $m_i^-$ for every $y\in \Xi(\omega,i).$
	\end{enumerate}
\end{definition}
The previous two theorems show that the regular words $\omega$ for the random walk form a full $\P$-measure subset of $\Sigma.$

\subsection{Basic properties of stationary measures}\label{sec: properties of stationary measures}
For the remainder of this section, we focus on the random walk induced by a finitely supported probability measure $\mu$ on $\Diff_+^{1+\alpha}(\SS^1)$ without common invariant probability measures on $\SS^1.$ 
We use the notation from the previous subsection.
Recall that $T_\mu$ denotes the semigroup generated by $\supp\mu.$ 
We first record several basic properties of $\mu$-stationary measures.

\begin{lemma}
\label{lem: not atomic}
Stationary measures are atomless.
\end{lemma}
\begin{proof}
Otherwise, for some $c>0,$ the atoms of $\nu$ with weights at least $c$ form a finite orbit of $T_\mu.$ Then, the uniform measure on this finite orbit is an invariant probability measure for $\supp\mu.$
\end{proof}

For $c > 0$, a subset of $\SS^1$ is said to be \textit{$c$-separated} if its elements are of distance strictly larger than $c$ one from each other.
\begin{lemma}
\label{lem: uniform bound of Pi}
There exists a constant $c>0$ depending only on $\mu$, such that for every $\omega$ regular for the random walk, $\Pi(\omega)$ and $\Xi(\omega)$ are $c$-separated.
\end{lemma}
\begin{proof}
By Lemma~\ref{lem: not atomic} and the compactness of $\SS^1,$ we can take $c>0$ such that for every interval $I$ on $\SS^1$ with length $|I| \leq c,$ it holds that
\[ \forall i\in\Zmodd,\quad \nu_i^+(I) < \frac{1}{r}.\]
Two consecutive points in $\Xi(\omega)$ bound an open interval $I$, which is a connected component of $W^s(\omega,i)$ for some $i \in \Zmodd$.
By Theorem~\ref{thm: structure of random walk 2}(5), $\nu_i^+(I) = 1/r$.
Hence $\abs{I} > c$.
We deduce that $\Xi(\omega)$ is $c$-separated. 
The proof for $\Pi(\omega)$ is similar.
\end{proof}

As a corollary of this lemma, the number of fixed points of a hyperbolic element in the semigroup $T_\mu$ is bounded from below. 
This will be helpful in showing that $d,r$ are topological invariants.
\begin{definition}\label{def: attractors and repellors}
Let $f\in \Homeo(\SS^1)$.
A fixed point $x$ of $f$ is called \textit{topologically hyperbolic} if there exists $\ve>0$ such that 
\begin{enumerate}
\item either for every $y\in B(x,\ve),$ $f^n(y)\to x(n\to+\infty),$
\item or for every $y\in B(x,\ve),$ $f^{-n}(y)\to x(n\to+\infty).$
\end{enumerate}
We call $x$ an \textit{attractor} of $f$ in the first case and \textit{repellor} otherwise.
\end{definition}
Observe that if every fixed point of $f\in\Homeo(\SS^1)$ is topologically hyperbolic, then necessarily, $f$ has finitely many fixed points and its attractors and repellors are arranged alternately. 

\begin{corollary}\label{cor: lower bound of fixed points}
	Let $f\in T_\mu$ be a diffeomorphism and assume that every fixed point is topologically hyperbolic. Then $f$ has at least $dr$ attractors and $dr$ repellors.
\end{corollary}
\begin{proof}
Assume for a contradiction that there exists $f\in T_\mu$ with $q$ attractors and $q$ repellors with $q<dr.$ 
Denote these fixed points by $a_1,r_1,a_2,r_2,\cdots,a_q,r_q$ ordered in the cyclic order and where $a_i$'s are the attractors and $r_i$'s are the repellors.
Take $\ve'>0$ such that for every interval $I \subset \SS^1$ of length $|I| \leq 2\ve',$ $\nu_i^+(I) < 1/qdr$ for every ergodic $\mu^+$-stationary measure. 
Consider the set $\Sigma_1=\lb{\omega^- \in \Sigma^-: \Pi(\omega^-)\cap(\bigcup_j B(r_j,\ve'))\ne\vn}.$ Then from Theorem~\ref{thm: structure of random walk 1}(2),
\[\P^-(\Sigma_1)\leq r\sum_{i \in \Zmodd} \sum_{j=1}^q \nu_i^+(B(r_j,\ve'))<1.\]
Let $\Sigma_2=\Sigma^-\sm\Sigma_1$ with positive $\P^-$ measure.
	
Assume that $f\in\cS^{*n}$ where $\cS=\supp\mu.$ Let $m$ be a positive integer large enough so that \[f^m\bigl(\SS^1\sm \bigcup_{j=1}^q B(r_j,\ve')\bigr)\sbs\bigcup_{j=1}^q B(a_j,c/2),\] where $c$ is the constant in Lemma~\ref{lem: uniform bound of Pi}. 
Consider the set
\[\Sigma'=\lb{\omega\in\Sigma:\pi^-\omega\in\Sigma_2,f_\omega^{nm}=f^m},\]
then $\P(\Sigma')=\P^-(\Sigma_2)  \mu^{*nm}(f^m)>0.$ But by the definition of $\Sigma_2$ and the cocycle invariance of $\Pi(\cd),$ for $\P$-almost every $\omega\in\Sigma',$ we have
\[\Pi(\sigma^{nm}\omega) = f^{m} \Pi(\omega) \sbs \bigcup_{j=1}^q B(a_j,c/2).\]
Since $q<dr$, $\Pi(\sigma^{nm}\omega)$ has two points within a distance of at most $c$. 
This contradicts Lemma~\ref{lem: uniform bound of Pi}.
\end{proof}

Now we prove Theorem \ref{thm: d,r top inv} assuming Lemma \ref{lem: 2dr fixed points}.

\begin{proof}[Proof of Theorem \ref{thm: d,r top inv}]
By Theorem \ref{thm: supports of stationary measures}, the constant $d$ is precisely  the number of $T_\mu$ minimal sets, which is obviously a topological invariant.
By Lemma~\ref{lem: 2dr fixed points} and Corollary~\ref{cor: lower bound of fixed points}, $2dr$ is the least number of fixed points of elements in $T_\mu$ having only topologically hyperbolic fixed points. 
Thus, $dr$ is also invariant under $C^0$ conjugates.
Hence, so is $r.$
\end{proof}

\subsection{Uniform good words}\label{sec: good words}

In this subsection, we will introduce a powerful tool: the set of uniform good words. 
Roughly speaking, it is a set of words with large probability and uniform control, and it has global contraction properties on each connected component of the $s$-manifolds $W^s(\omega)$. 
It enables us to handle smooth actions in a manner similar to linear actions, except that in the smooth case there are $dr$ cones instead of one cone in the linear case.

\begin{proposition}\label{prop: good words}
	For any $\ve>0,$ there exists a subset $\Sigma_\ve\sbs\Sigma$ satisfying the following.
	\begin{enumerate}
		\item $\P(\Sigma_\ve)>1-\ve$.
		\item For every $(\omega,i) \in \Sigma_\ve\times [d]$ and every $x\in\Pi(\omega,i)$, we have  $\lim_{n\to+\infty}\frac{1}{n}\log (f_\omega^n)'(x)=\lambda_i^+$. The convergence is uniform in $\omega$,$i$ and $x$.
		\item there exist a constant $C=C(\ve)>0$ and an open set $U(\omega)=U(\omega,\ve)$ for every $\omega\in\Sigma_\ve,$ such that
		\begin{itemize}
			\item $\Pi(\omega)\sbs U(\omega)\sbs W^s(\omega),$
			\item $\SS^1\sm U(\omega)\sbs \Xi(\omega)^{(\ve)},$ where $\Xi(\omega)^{(\ve)}=\bigcup_{x\in \Xi(\omega)}B(x,\ve).$
			\item for all $n\geq 0$, $\vk_\alpha (f_\omega^n,I)\leq C(\ve)$ for every connected component $I$ of $U(\omega,\ve).$
		\end{itemize}
	\end{enumerate}
\end{proposition}
We call the set $\Sigma_\ve\sbs\Sigma$ constructed above a set of \textit{uniform good words}.

\begin{lemma}\label{lem: Pesin good distortion estimate}
	Let $\Sigma'\sbs\Sigma$ be a subset and $x\in\SS^1.$ Assume there exists $C>0,c>0$ such that for every $\omega\in\Sigma',$ $(f_\omega^n)'(x)\leq C 2^{{-c} n}.$ 
	Then there exist constants $C',\rho>0$ depending only on $C$ and $c$ such that
	\[\vk_\alpha(f_\omega^n,B(x,\rho))\leq C' ,\quad \forall \omega\in\Sigma',\ n\in\NN. \]
\end{lemma}

\begin{proof}
	By the proof of Lemma \ref{lem: expdecay f'}, there is $\rho>0$ depending only on $C>0,c>0$ such that
	\[ (f_\omega^n)'(y)\leq 2C \cdot 2^{{-c} n},\quad\forall y\in B(x,\rho), n\geqslant 0.\]
	Then for every subinterval $I\sbs B(x,\rho),$ we have
	\[\vk(f_\omega^n, I)\leqslant  \sum_{k=0}^{n-1}\vk(f_{k},f_\omega^k(I))\leqslant \sum_{k=0}^{n-1}\vk_\alpha(f_{k})\cdot |f_\omega^k(I)|^\alpha\leqslant  (2C)^\alpha M\cdot |I|^\alpha \sum_{k=0}^{n-1} 2^{-c\alpha k},\]
	where $M=\max_{f\in\cS}\vk_\alpha(f)$ is a uniform constant. 
	Then
	\[\vk_\alpha(f_\omega^n, B(x,\rho))=\sup_{I\ne\vn\sbs B(x,\rho)}\frac{\vk(f_\omega^n,I)}{|I|^\alpha}\leqslant (2C)^\alpha M \sum_{k=0}^{\infty} 2^{-c\alpha k}. \]
	The conclusion follows by taking $C'=(2C)^\alpha M \sum_{k=0}^{\infty} 2^{-c\alpha k}.$
\end{proof}

\begin{proof}[Proof of Proposition \ref{prop: good words}]
First, we can take $\Sigma'\sbs\Sigma$ with $\P(\Sigma')>1-\ve/2$ such that 
\[\forall \omega\in\Sigma',\, \forall i\in\Zmodd,\, \forall x\in\Pi(\omega,i),\quad \lim_{n\to+\infty}\frac{1}{n}\log (f_\omega^n)'(x)=\lambda_i^+,\] 
where the convergence is uniform. Let 
\[c=\frac{1}{2}\inf_{i\in\Zmodd}-\lambda_i^+>0,\]
	then there exists $C>0$ such that $\forall\omega\in\Sigma',$
	\[(f_\omega^n)'(x)\leq C2^{-c n},\quad\forall x\in\Pi(\omega),n\in\NN.\]
	By Lemma \ref{lem: Pesin good distortion estimate}, there exists $C',\rho>0$ such that for every $\omega\in\Sigma',$
	\[\vk_\alpha(f_\omega^n,B(x,\rho))\leq C',\quad \forall x\in\Pi(\omega),n\in\NN.\]
	In particular, for every $y\in B(x,\rho),$ $\frac{1}{n}\log(f_\omega^n)'(y)\to\lambda_i^+<0,$ it follows that $y\notin \Xi(\omega).$
	
	For every $\omega\in\Sigma',$ let
	\[K(\omega)=\SS^1\sm\sb{\bigcup_{x\in \Pi(\omega)}B(x,\rho)}.\]
	Then $\Xi(\omega)\sbs K(\omega)\sbs W^u(\omega)$ and $K(\omega)$ is a finite union of closed intervals. By (7) of Theorem \ref{thm: structure of random walk 2}, for every connected component $J$ of $K(\omega),$ $|f_\omega^{-n}J|\to 0.$ 
	We can find a subset $\Sigma''\sbs\Sigma'$ with $\P(\Sigma'')>1-\ve$ and a sufficiently large integer $N$ satisfying
	\[|f_\omega^{-N}J|<\ve\]
	for every $\omega\in\Sigma''$ and every connected component $J$ of $K(\omega).$

	Fix a choice of $\Sigma''$ and $N$ for a given $\ve>0.$
	Let $\Sigma_\ve=\sigma^{-N}\Sigma'',$ then $\P(\Sigma_\ve)=\P(\Sigma'')>1-\ve$ and the condition (2) holds. For every $\omega\in \Sigma_\ve,$ set
	\[U(\omega)=f_{\sigma^N\omega}^{-N}\sb{\bigcup_{x\in\Pi(\sigma^N\omega)}B(x,\rho)},\]
	so that $\Pi(\omega)\sbs U(\omega)\sbs W^s(\omega)=\SS^1\sm\Xi(\omega).$ The complement of $U(\omega)$ is exactly
	\[f_{\sigma^N\omega}^{-N}(K(\sigma^N\omega)),\]
whose connected components have length less than $\ve$ by the choice of $N.$ 
It remains to estimate the distortion on $f_{\sigma^N\omega}^{-N} B(x,\rho).$ 
Because only finitely many different diffeomorphisms $f_\omega^N$ occur, we can choose a constant $M>0$ so that for every $\omega\in \Sigma,$
\[\| (f_\omega^N)'\|\leqslant M\quad\text{and}\quad \vk_\alpha(f_\omega^N,\SS^1 )\leqslant M. \]
Then for $n\leqslant N,$ we get $\vk_\alpha(f_\omega^n,f_{\sigma^N\omega}^{-N} B(x,\rho)) \leqslant M.$
For $n>N,$ recall that $\sigma^N\omega\in\Sigma''.$ 
In view of \eqref{eqn: vk alpha fg}, for every $ x\in \Pi(\sigma^N\omega)$ and $ n>N$, we have
	\[\vk_\alpha(f_\omega^n,f_{\sigma^N\omega}^{-N} B(x,\rho))\leqslant \|(f_\omega^N)' \|^\alpha \vk_\alpha(f_{\sigma^N\omega}^{n-N},B(x,\rho) )+\vk_\alpha(f_\omega^N ,f_{\sigma^N\omega}^{-N} B(x,\rho))\leqslant M^\alpha C'+M. \]
This finishes the proof of the proposition by taking $C(\ve)=M^\alpha C'+M.$
\end{proof}

\begin{corollary}\label{cor: contraction of intervals}
	For every $\ve>0,$ let $\Sigma_\ve\sbs\Sigma$ be the set defined in Proposition \ref{prop: good words} with the corresponding $U(\omega)=U(\omega,\ve).$ 
	Then there exists a sequence of positive numbers $\ve_n\to 0$ such that
	\[|f_\omega^nI|<2^{(\lambda_i^++\ve_n)n}\]
	for every $\omega\in\Sigma_\ve,n\in\NN$ and a connected component $I$ of $U(\omega)\cap W^s(\omega,i).$
\end{corollary}

Combined with Theorem~\ref{thm: structure of random walk 2}(5) and Lemma~\ref{lem: not atomic}, the following is immediate.
\begin{corollary}
\label{cor: contraction of intervals 2}
There exists $\Sigma' \sbs \Sigma$, such that $\P(\Sigma')>0$ and the following holds for every $\omega \in \Sigma'$. 
For every $i \in \Zmodd$ and every $x \in \Pi(\omega,i)$ there is an open interval $I$ containing $x$ satisfying $\nu_i^+(I) \geq \frac{1}{2r}$ and such that
\[\forall n \in \NN,\quad  \lambda^+_i - \ve_n \leq \frac{1}{n} \log |f_\omega^n I| \leq  \lambda^+_i + \ve_n\]
where $\ve_n \to 0^+$ is a sequence independent of $(\omega,x)$.
\end{corollary}

\subsection{Effective convergence to the Furstenberg boundary}\label{sec: effective convergence}
We may consider the distribution of $\Pi(\omega,i)$ as a probability measure on $\SS_r.$ For every $i\in\Zmodd,$ consider the measurable maps
\[\Pi(\cd,i):\Sigma\to\SS_r,\quad \Xi(\cd,i):\Sigma\to\SS_r\]
Let $\ul\nu_i^+,\ul\nu_i^-$ be the probability measures on $\SS_r$ which are the pushforwards of $\P$ by $\Pi(\cd,i),\Xi(\cd,i),$ respectively. 
For every $\ul x\in\SS_r,$ recall $u_{\ul x}$ denotes the uniform probability measure on $\ul x\sbs\SS^1.$ Combining Theorem \ref{thm: structure of random walk 1}(2) and the definition of $\ul\nu_i^\pm,$ we have
\[\nu_i^+=\int_{\SS_r} u_{\ul x}\dr \ul\nu_i^+(\ul x), \quad \nu_i^-=\int_{\SS_r} u_{\ul x}\dr \ul\nu_i^-(\ul x).\]

\begin{proposition}
\label{prop: convergence in probability}
Given $\ve>0$, there exists $N>0$ such that for every $i\in\Zmodd$, for all $\ul x \in \supp\ul\nu_i^+$ and all $n 
 \geq N$, 
\[\P \lb{\omega:\underline d(f_{\sigma^{-n}\omega}^n \ul x,\Pi(\omega,i))<2^{(\lambda_i^+ +\ve)n} }\geq 1-\ve.\]
\end{proposition}
This proposition can be interpreted as an effective version of the convergence~\eqref{eqn: F boundary}.

For $\omega \in \Sigma$ and $i \in \Zmodd$, define 
\[\ul W^s(\omega,i) = \Bigl\{\, \ul x \in \SS_r: \bigcup_{x\in \ul x} W^s(\omega,x) = W^s(\omega,i) \,\Bigr\}.\]
In other words, $\ul W^s(\omega,i)$ is the set of $\ul x \in \SS_r$ whose elements fall in different connected components of $W^s(\omega,i)$. 
Note that $\ul W^s(\omega,i)$ depends only on $(\omega^+,i)$ and that $\Pi(\omega,i) \in W^s(\omega,i)$ for $\P$-almost every $\omega \in \Sigma$.
Hence for every $\ul x\in\supp\ul\nu_i^+$, we have $\ul x\in\ul W^s(\omega,i)$ for $\P$-almost every $\omega$.

Indeed, we will show Proposition~\ref{prop: convergence in probability} for all $\ul x \in \SS_r$ such that $\ul x \in \ul W^s(\omega,i)$ for $\P$-almost every $\omega \in \Sigma$. Such $\ul x$ form a larger set than $\supp \ul\nu_i^+$ in some situations.
For instance, if $d = r =1$, then the proposition holds for all $x \in \SS^1$. 

\begin{proof}
In view of Lemma~\ref{lem: not atomic}, let $\ve'\in (0,\ve/2)$ be such that 
\[\forall j \in \Zmodd,\, \forall x \in \SS^1,\quad \nu_j^-(B(x,\ve')) \leq \frac{\ve}{2dr^2}.\]
Let $\Sigma_{\ve'} \sbs\Sigma$ be the set of uniform good words with the corresponding constant $C=C(\ve')$ and open sets $U(\omega)=U(\omega,\ve')$ given by Proposition~\ref{prop: good words}. 
Consider 
\[\Sigma' = \lb{\,\omega \in \Sigma_{\ve'} : \ul x \sbs U(\omega)\,}.\] 
We claim that $\P(\Sigma') \geq 1-\ve.$
Indeed, if $\omega \in \Sigma_{\ve'}$ and $\omega \not\in \Sigma'$, then there exists $x \in \ul x$ such that $x \in  \SS^1 \sm U(\omega) \sbs \bigcup_{y \in \Xi(\omega)} B(y,\ve')$.
It follows that $B(x,\ve') \cap \Pi(\omega) \ne \vn.$ 
Recalling (1) and (2) in Theorem~\ref{thm: structure of random walk 1}, we have
\[\P(\Sigma_{\ve'} \sm \Sigma') \leq 
\int r\sum_{j \in \Zmodd} \sum_{x \in \ul x} \unif_{\Xi(\omega,i)} ( B(x,\ve')) \dd \P(\omega) = r\sum_{j \in \Zmodd} \sum_{x \in \ul x} \nu_j^-(B(x,\ve')) \leq r^2d\cdot\frac{\ve}{2dr^2}= \frac{\ve}{2}.\]
This proves that $\P(\Sigma') \geq 1-\ve.$

Moreover, for $\P$-almost every $\omega\in\Sigma'$, since $\ul x \in \ul W^s(\omega,i)$, each element of $\ul x$ falls exactly in one of the $r$ connected components of $U(\omega)$.
By Corollary \ref{cor: contraction of intervals}, there exists $N$ depending only on $\mu$ and $\ve$ such that for all $n \geq N$, all $\omega\in\Sigma'$ and every connected component $I$ of $U(\omega),$ $|f_\omega^nI|<\frac{1}{r}2^{(\lambda_i^+ +\ve)n}.$ 

Now for every $n \geq N$, if $\sigma^{-n} \omega\in\Sigma',$ then elements of $\ul x$ can be paired with those of $\Pi(\sigma^{-n}\omega,i)$ such that each pair belongs to the same connected component of $U(\sigma^{-n}\omega)$. Hence
\[\underline d(f_{\sigma^{-n}\omega}^n \ul x,\Pi(\omega,i))<r\cdot \frac{1}{r}2^{(\lambda_i^+ +\ve)n}=2^{(\lambda_i^+ +\ve)n}.\]
Since $\P(\sigma^{-n}\Sigma')=\P(\Sigma')\geq 1-\ve,$ the conclusion follows.
\end{proof}

\subsection{Global contraction for group actions}
In this subsection, we establish a global contraction property that will be used in Part II.
Let $G\sbs \Diff_+^{1+\alpha}(\SS^1)$ be a finitely generated subgroup of circle diffeomorphisms.
It is well known that if $G$ has no finite orbit, then $G$ has a unique minimal set on $\SS^1,$ which is either $\SS^1$ or a Cantor set.
In this situation, we may choose a finitely supported probability measure $\mu$ on $G$ such that $\supp \mu$ generates $G$ as a semigroup.
Thus $T_\mu=G.$
If we additionally assume that $G$ preserves no common invariant probability measure on $\SS^1,$ then Theorem \ref{thm: supports of stationary measures} asserts that $\mu$ has a unique stationary measure, whose support is the unique $G$-minimal set $\Lambda.$
Following the notation of Theorem \ref{thm: structure of random walk 1}, we have $d=1.$

\begin{lemma}\label{lem: contracting constant}
Let $G\sbs \Diff_+^{1+\alpha}(\SS^1)$ be a finitely generated subgroup without invariant probability measures on $\SS^1$. 
Let $\Lambda \sbs \SS^1$ be the unique minimal set of $G.$ 
Then there exists $\ve_1>0$ such that for any $x,y\in\Lambda$ and any $\ve> 0$ there exists $f\in G$ such that $f(B(x,\ve_1))\sbs B(y,\ve).$
\end{lemma}
\begin{proof}
Let $\cG$ be a finite symmetric generating set of $G,$ and let $\mu$ be the uniform probability measure on $\cG$. 
It induces a random walk on $\SS^1$. 
By Theorem~\ref{thm: supports of stationary measures}, as $\Lambda$ is the unique minimal set of the semigroup generated by $\supp\mu=\cG,$ there exists a unique $\mu$-stationary measure $\nu$, whose support is $\Lambda.$ 
Since $\mu$ is symmetric,  $\nu$ is also the unique $\mu^{-}$-stationary measure. 
Recall the constant $r$ and $\Xi(\omega^+)$ in Theorem \ref{thm: structure of random walk 1}. 

In view of Lemma~\ref{lem: not atomic}, we take $\ve_1>0$ such that 
$\sup_{z \in \SS^1} \nu(\ol{B(z,\ve_1)}) <1/2r.$
Recall that, since $d=1$ and $\mu^-=\mu$, Theorem~\ref{thm: structure of random walk 1}(2) gives
\[
\nu=\int_{\Sigma^+}\unif_{\Xi(\omega^+)} \dd\P^+(\omega^+).
\]
Consequently,
\[\P^+\lb{\omega^+\in\Sigma^+:\ol{ B(x,\ve_1)} \cap \Xi(\omega^+)\ne\vn}\leq r  \nu(\ol{ B(x,\ve_1)})<\frac{1}{2}.\]
In view of Theorem~\ref{thm: structure of random walk 2}, the set of $\omega^+ \in \Sigma^+$ such that $\ol{B(x,\ve_1)} \sbs W^s(\omega^+)$ has positive $\P^+$-measure.
For such $\omega^+$, we have $|f_{\omega^+}^n B(x,\ve_1)|\to 0$ as $n\to+\infty.$

Since $x \in \Lambda=\supp\nu,$ we have $\nu(B(x,\ve_1))>0$ so that there is $\omega^+$ as above and $x'\in B(x,\ve_1)$ such that $(\omega^+,x')$ is generic for the
measure preserving system $(\Sigma^+\times\SS^1,\P^+\times\nu,F^+).$ 
Moreover, $\nu(B(y,\ve/2))>0$ since $y \in \Lambda=\supp\nu$.
Thus, by Birkhoff's ergodic theorem, there are arbitrarily large integers $n$ such that $f_{\omega^+}^n(x')\in B(y,\ve/2)$. 
For $n$ large enough, we also have $|f_{\omega^+}^n(B(x,\ve_1))|<\ve/2.$ 
Set $f=f_{\omega^+}^n\in G.$ 
Then $f(B(x,\ve_1))\sbs B(y,\ve).$
\end{proof}

\section{Exact dimensionality of stationary measures}\label{se:5}
In this section, we will prove Theorem \ref{thm: exact dimensionality}, which gives the exact dimensionality of the ergodic stationary measures. 
Our strategy is similar to that used in \cite[Theorem 3.4]{HS17} for $\PSL(2,\RR)$ actions, but general circle diffeomorphisms are not as rigid as M\"obius transformations. 
Therefore, we will make use of Pesin theory, distortion control, and a hyperbolic time argument (the set given by \eqref{eqn: hyperbolic times} corresponds to hyperbolic times) to obtain our result.

\subsection{Preparation}
Here, we list some useful results taken from~\cite{HS17}.

Let $\mu$ be a finitely supported probability measure on $\Diff_+^{1+\alpha}(\SS^1)$ without common invariant probability measures. 
Let $\nu$ be a $\mu$-stationary measure on $\SS^1$ and let $m$ be the corresponding ergodic $u$-state (recall Proposition~\ref{prop: stationary and state}).
For every $f \in \cS$ and interval $I\sbs\SS^1$, define
\begin{equation}\label{eqn: definition of phi_f}
	\vp(f,I)\defeq\frac{f_*\nu(I)}{\nu(I)}.
\end{equation}
Besicovitch differentiation theorem states that
\[\frac{\dr f_*\nu}{\dr \nu}(x)=\lim_{I\ni x,\ |I|\to 0}\vp(f,I),\quad\ae[\nu\text{-}]x.\]

\begin{lemma}[{\cite[Lemma 3.5]{HS17}}]
For every $f\in\cS$ and $t>0,$
\[\nu\lb{x:\sup\nolimits_{I\ni x}\vp(f,I)>t}\leq 2t^{-1},\quad f_*\nu\lb{x:\inf\nolimits_{I\ni x}\vp(f,I)<t^{-1}}\leq 2t^{-1}.\]
\end{lemma}
Using the relation $Q_*m = \nu$ and the fact that $m$ is a $u$-state, we obtain the following consequence, which is essentially \cite[Corollary 3.6]{HS17}. 
\begin{lemma}
\label{lem: m phi too big}
For every $t>0,$
\[m\lb{(\omega,x) \in \Sigma \times \SS^1 :\sup\nolimits_{I\ni x} \vp(f_{\sigma^{-1}\omega},I) \geq t} \ll_r t^{-1},\]
and
\[m\lb{(\omega,x) \in \Sigma \times \SS^1 :\inf\nolimits_{I\ni x} \vp(f_{\sigma^{-1}\omega},I) \leq t^{-1}} \ll_r t^{-1}.\]
\end{lemma}
Define $\cO:\Sigma \times \SS^1 \times\RR_+\to \RR_+$ by
\[\cO(\omega,x,\rho)\defeq \sup_{I\ni x,|I|\leq\rho} \log\vp(f_{\sigma^{-1}\omega},I)-\inf_{I\ni x,|I|\leq\rho} \log\vp(f_{\sigma^{-1}\omega},I).\]
Besicovitch differentiation theorem can be restated as $\lim_{\rho\to 0^+}\cO(\omega, x, \rho)=0$ for $m$-almost every $(\omega,x) \in \Sigma \times \SS^1$.
By Lemma~\ref{lem: m phi too big} and the layer cake representation, the following is immediate.
\begin{lemma}[see {\cite[Corollary 3.7]{HS17}}]
The function $(\omega,x) \mapsto\sup_{\rho>0}\cO(\omega,x,\rho)$ is in $L^1(\Sigma \times \SS^1,m).$
\end{lemma}

The following is a variant of Maker's Theorem, as stated in \cite{HS17}.
\begin{theorem}[{\cite[Theorem 3.8]{HS17}}]\label{thm: Maker's Theorem}
	Let $(X,\cF,\theta,T)$ be an ergodic probability measure preserving system. 
	Let $G_t:X\to\RR$ be a measurable $1$-parameter family of measurable functions such that $\sup_t|G_t|\in L^1.$ 
	Suppose that 
	\[G=\lim_{t\to 0}G_t\]
	exists almost everywhere. 
	Let $t_{N,n}:X\to\RR$ be functions with the property that for $\theta$-a.e. $x$ and every $\ve>0,$ for large enough $N,$
	\[|t_{N,n}|<\ve,\quad \text{for }1\leq n\leq (1-\ve)N. \]
	Then
	\[\lim_{N \to +\infty}\frac{1}{N}\sum_{n=1}^N G_{t_{N,n}(x)}(T^nx)=\int G \dd\theta,\quad \ae[\theta\text{-}] x.\]
\end{theorem}

\subsection{Exact dimensionality (Proof of Theorem \ref{thm: exact dimensionality})}
\label{subsec: proof of exact dimensionality}

Let $\nu$ be an ergodic $\mu$-stationary measure and let 
$m \in \cP^u$ be the corresponding ergodic $u$-state. 
Let $\lambda = \lambda(\mu,\nu) = \lambda(m)$ denote the Lyapunov exponent, which is negative under our assumption.
Let $h=\hFur(\mu,\nu)$ be the Furstenberg entropy.
Using the relation between $\nu$ and $m$ and a change of variable, note that the Furstenberg entropy can be expressed as
\begin{equation*}
h = \int \log \frac{\dr (f_{\sigma^{-1}\omega})_* \nu}{\dr \nu}(x) \dd m(\omega,x).
\end{equation*}

\begin{proof}[Proof of Theorem \ref{thm: exact dimensionality}]
We aim to show that for $\nu$-almost every $x\in \SS^1$, or, in other words, for $m$-almost every $(\omega,x) \in \Sigma \times \SS^1$, 
\begin{equation}
\label{eqn: local dim at x}
\lim_{\rho \to 0^+}\frac{\log\nu(B(x,\rho))}{\log\rho}=-\frac{h}{\lambda}. 
\end{equation}

Fix an arbitrary element $\omega\in\Sigma$ regular for the random walk (Definition \ref{def: regular for random walk}) and a point $x\in\supp m_\omega.$ Let $I_0 = B(x,\rho)$ and for integers $n \geq 0$, set 
\[x_n = f_\omega^{-n} x \quad\text{and}\quad I_n = f_\omega^{-n}I_0.\]
For any integer $N \geq 1$, we can write
\[\nu(I_0)=\nu(I_N)\prod_{n=0}^{N-1}\frac{\nu(I_{n})}{\nu(I_{n+1})}=\nu(I_N)\prod_{n=0}^{N-1}\frac{\nu(I_n)}{(f_{\sigma^{-(n+1)}\omega})_*\nu(I_n)}.\]
Recall the definition of $\vp(f,I)$ in \eqref{eqn: definition of phi_f}, 
\[-\log \nu(I_0)=-\log\nu(I_N)+\sum_{n=0}^{N-1}\log\vp(f_{\sigma^{-(n+1)}\omega},I_n).\]
Since
\[\abs{\log\frac{\dr (f_{\sigma^{-(n+1)}\omega})_*\nu}{\dr \nu}(x_n)-\log\vp(f_{\sigma^{-(n+1)}\omega},I_n) }\leq\cO(\sigma^{-n}\omega,x_n,|I_n|),\]
we have
\begin{equation}
\label{eqn: nuI0 and sum}
\abs{-\log\nu(I_0)- \sum_{n=0}^{N-1}\log\frac{\dr (f_{\sigma^{-(n+1)}\omega})_*\nu}{\dr \nu}(x_n)}\leq -\log\nu(I_N)+\sum_{n=0}^{N-1}\cO(\sigma^{-n}\omega,x_n,|I_n|).
\end{equation}
Note that the sum in the left-hand side is independent of $\rho$ and it is a Birkhoff sum of the function $(\omega,x) \mapsto \log\frac{\dr(f_{\sigma^{-1}\omega})_*\nu }{\dr\nu}(x)$ for the transformation $F^{-1}$.
Hence, by Birkhoff's ergodic theorem, outside a null set, as $N \to +\infty$,
\begin{equation}
\label{eqn: average to hFur}
\frac{1}{N}\sum_{n=0}^{N-1}\log\frac{\dr (f_{\sigma^{-(n+1)}\omega})_*\nu}{\dr \nu}(x_n)\to h.
\end{equation}

Recall that our objective is to estimate $\lim_{\rho \to 0^+}\frac{\log \nu(I_0)}{\log \rho}$. 
To make use of \eqref{eqn: nuI0 and sum}, we still have the freedom to choose $N = N(\omega,\rho)$ according to $\rho > 0$.
This choice will guarantee $N(\omega,\rho) \sim\frac{\log\rho}{\lambda}$ while simultaneously ensuring that the right-hand side of \eqref{eqn: nuI0 and sum} is negligible compared to $N(\omega,\rho)$ as $\rho \to 0^+$.

By Corollary~\ref{cor: contraction of intervals 2}, we first  fix a subset $\Sigma' \sbs \Sigma$ and a sequence $\ve_n \to 0^+$, such that $\P(\Sigma') > 0$ and that for $m$-almost every $(\omega',y) \in \Sigma' \times \SS^1$, there is an interval $I = I(\omega',y)$ containing $y$ and satisfying $\nu(I) \geq \frac{1}{2r}$ and for any $n \in \NN$, $\log |f_{\omega'}^n I| \leq (\lambda + \ve_n) n$.

For every $\omega\in\Sigma,$ we define
\begin{equation}
\label{eqn: hyperbolic times}
	A_\omega\defeq \lb{N : \sigma^{-N}\omega \in \Sigma'},
\end{equation}
which is the set of \textit{hyperbolic times} at which $f_{\sigma^{-N}\omega}^N$ exhibits substantial contraction for $N\in A_\omega$. For $\P$-almost every $\omega \in \Sigma$, the set $A_\omega$ is infinite and thus,
\[N(\omega,\rho) \defeq \min\lb{ N \in A_\omega : 2^{(\lambda + \ve_N)N} \leq \rho}\]
 is well defined for all $\rho > 0$.
We claim that
\begin{equation}
\label{eqn: rho and N}
\frac{\log \rho}{N(\omega,\rho)} \to \lambda \quad \text{as } \rho \to 0^+.
\end{equation}
Indeed, if we list elements of $A_\omega = \{N_1, N_2, \dotsc \}$ in increasing order, then by Birkhoff's ergodic theorem, for $\P$-almost every $\omega$, $N_k/k \to \P(\Sigma')^{-1}$ and hence $N_{k-1}/N_k \to 1$ as $k\to +\infty$. 
By definition, $N(\omega,\rho)$ is some $N_k$ such that $(\lambda + \ve_{N_k})N_k \leq \log \rho < (\lambda + \ve_{N_{k-1}})N_{k-1}$.
Dividing by $N_k$ and taking limits proves \eqref{eqn: rho and N}.

We write $N = N(\omega,\rho)$ as shorthand.
From $N \in A_\omega$ and the point $x \in \supp m_\omega$, we know that there is an interval $I = I(F^{-N}(\omega,x))$ containing $f_\omega^{-N} x$ and satisfying $\nu(I) \geq \frac{1}{2r}$ and 
\[  |f_{\sigma^{-N}\omega}^N I| \leq 2^{(\lambda + \ve_N)N} \leq \rho.\]
Hence $f_{\sigma^{-N}\omega}^N I \subset B(x,\rho) = I_0$ and $I_N = f_\omega^{-N} I_0 \sps I$ has measure
\begin{equation}
\label{eqn: nu of IN}
\nu(I_N) \geq \frac{1}{2r}.
\end{equation}

We claim that for $\P$-almost every $\omega \in \Sigma$, there is $\wt\ve_\rho \to 0$ as $\rho \to 0^+$ (depending on $\omega$) such that for every $1\leq n\leq N$ we have
\[|I_n| \leq 2^{\lambda(N-n)+\wt\ve_\rho N}. \]	
Thus, applying Theorem \ref{thm: Maker's Theorem} (let $(G_t, t_{N,n}, T, X)$ be $(\cO(\cdot, t), |I_n|, F^{-1}, \Sigma \times \SS^1)$) 
we obtain that
\[\frac{1}{N}\sum_{n=0}^{N-1}\cO(\sigma^{-n}\omega,x_n,|I_n|)=\frac{1}{N}\sum_{n=0}^{N-1}\cO(F^{-n}(\omega,x),|I_n|)\to 0\]
for $m$-almost every $(\omega,x) \in \Sigma \times \SS^1$.
Combining~\eqref{eqn: nuI0 and sum}, \eqref{eqn: average to hFur}, \eqref{eqn: nu of IN} and \eqref{eqn: rho and N}
we obtain the desired convergence \eqref{eqn: local dim at x}.

To prove the claim, fix $\ve>0$. By regularity, there exists
$C_\ve\geq1$ such that
\[
(f_\omega^{-n})'(x)
\leq C_\ve 2^{(-\lambda+\ve)n},
\qquad n\geq0.
\]
Since $\supp m_\omega$ is finite, we may choose $C_\ve$
uniformly over $x\in\supp m_\omega$.

Let $M=\max_{f\in\cS}\vk_\alpha(f^{-1})$, and choose $0<\gamma<1$ such that $M\gamma^\alpha\leq\ve$.
For $\rho<\gamma /2$, let $k$ be the first integer such that
$|I_k|>\gamma $, with $k=\infty$ if no such integer exists.
For every $n\leq k$, the distortion composition rule gives
\[\vk(f_\omega^{-n},I_0)\leq M\sum_{j=0}^{n-1}|I_j|^\alpha\leq nM\gamma^\alpha\leqslant n\ve .\]
Consequently, for every $n\leqslant k,$ $|I_n|\leq2C_\ve\rho\,2^{(-\lambda+2\ve)n}.$
In particular, if $k<\infty,$ the case $n=k$ leads to the estimate
\[2C_\ve\rho\,2^{(-\lambda+2\ve)k}\geqslant |I_k|\geqslant\gamma.\]
Thus,  for every $n\geqslant 0,$
\[|I_n|\leq \frac{2C_\ve}{\gamma}\rho\,2^{(-\lambda+2\ve)n}\ll_\ve \rho 2^{(-\lambda+2\ve)n},\]
By \eqref{eqn: rho and N}, we have $\rho=2^{\lambda N+o(N)}$.
Hence, uniformly for $0\leq n\leq N$,
\[ \log|I_n| \leqslant \lambda(N-n)+2\ve N+o(N).\]
Since $\ve>0$ is arbitrary, the claim follows.
\end{proof}

\section{Dimension formula}\label{se:6}
In the previous section, we established the exact dimensionality of ergodic stationary measures; its dimension is expressed as the ratio between the Furstenberg entropy and the Lyapunov exponent.
However, computing the Furstenberg entropy can be challenging in practice.
In this section, we will demonstrate that the random walk entropy can be used as a substitute for the Furstenberg entropy under a suitable separation condition, namely the local discreteness of the group generated by the support of the measure.

\subsection{Dimension and entropy}
Let $\nu$ be a Borel probability measure on $\SS^1$, and let $\cA$ be a finite Borel partition of $\SS^1.$ 
Recall the \textit{Shannon entropy} of $\nu$ with respect to $\cA$ is 
\[H(\nu,\cA)\defeq -\sum_{A\in\cA}\nu(A)\log\nu(A).\]
Note that $H(\nu,\cA)\leq\log\#\cA.$ Let $\cB$ be another finite measurable partition of $\SS^1,$ the \textit{conditional entropy} is 
\[H(\nu,\cA|\cB)\defeq-\sum_{B\in\cB}\sum_{A\in\cA}\nu(A\cap B)\log\frac{\nu(A\cap B)}{\nu(B)}.\]
Similarly, an upper bound for the conditional entropy is given by
\begin{equation}
\label{eqn: bound for cond entropy}
H(\nu,\cA|\cB)\leq \max_{B\in\cB}\log\#\lb{A\in\cA:\nu(A\cap B)>0}.
\end{equation}

Fix a finite measurable partition $\cA$. The function $\nu \mapsto H(\nu,\cA)$ is concave and almost convex: 
for any probability vector $\bm{p}=(p_1,\cdots,p_k)$ and any Borel probability measures $\nu_1,\dotsc,\nu_k$, we have
\begin{equation}
\label{eqn: entropy concave}
\sum_{i=1}^k p_i H(\nu_i,\cA) \leq H\sb{\sum_{i=1}^k p_i\nu_i,\cA}\leq \sum_{i=1}^k p_i H(\nu_i,\cA)+H(\bm{p}).
\end{equation}
where $H(\bm{p})=-\sum p_i\log p_i$.

We often consider entropy with respect to dyadic partitions on $\SS^1$.
We identify $\SS^1$ with $[0,1[$. 
For a positive integer $n$, let
\[\cD_n\defeq\lb{{\left[\frac{k}{2^n},\frac{k+1}{2^n}\right[}: 0\leq k\leq 2^n-1}.\]
Moreover, for a positive number $t,$ let $\cD_t= \cD_{\lfloor t\rfloor}.$
If the sequence $(\frac{1}{n}H(\nu,\cD_n))_n$ converges and
\[\lim_{n\to+\infty}\frac{1}{n}H(\nu,\cD_n)=\beta,\]
we say that $\nu$ has \textit{entropy dimension} $\beta$.
The following standard proposition is proved in  \cite[Theorem 4.4]{Y82}; see also~\cite[Theorem 1.3]{FLR}.
\begin{proposition}
	If $\nu$ is exact dimensional then the entropy dimension of $\nu$ exists and coincides with its exact dimension.
\end{proposition}

\subsection{Dimension formula on the circle}\label{subsec: dimension formula}

Let $\mu$ be a probability measure on $\Diff_+^{1+\alpha}(\SS^1)$ with finite support such that $\supp\mu$ does not preserve any Borel probability measure on $\SS^1$.
Let $\nu$ be an ergodic $\mu$-stationary measure on $\SS^1.$
We know that $\nu$ is exact dimensional by Theorem \ref{thm: exact dimensionality}, and we denote its dimension by $\beta$.

Let us first explain the idea of the proof. 
As explained in the previous subsection, the entropy dimension of $\nu$ coincides with its exact dimension.
Proposition \ref{prop: convergence in probability} gives rise to some $\ul x \in \SS_r$ for which we can write $\nu$ as the weak-* limit 
\[\nu=\lim_{n\to\infty}\mu^{*n}* u_{\ul x}.\]
Combining the convergences above, we can expect that
\[\frac{1}{n}H(\mu^{*n}*u_{\ul x},\cD_{-\lambda n})\to -\lambda\beta,\]
where $\lambda=\lambda(\mu,\nu)<0$ is the corresponding Lyapunov exponent. 
It relates the action of $n$-step iterations of $\mu$ and the dimension $\beta.$

If there were a gap between $h_{\mr{RW}}(\mu)$ and $-\lambda \beta,$ then there would be exponentially many elements in $\supp\mu^{*n}$ mapping a point in $\ul x$ into the same atom of $\cD_{-\lambda n}.$ 
Combining the derivative and distortion controls for good words in $\supp\mu^{*n}$ (see Proposition \ref{prop: good words}), the restrictions of these elements to a fixed interval would be close to each other. 
This allows us to construct elements of the form $g^{-1} f$ that are arbitrarily $C^1$-close to the identity on an interval. 
This contradicts the local discreteness assumption.

Now we start the proof.
We use the notation of Theorem~\ref{thm: structure of random walk 1} and without loss of generality, assume that $\nu=\nu_0^+$.
Recall $\ul\nu_0^+$ is the probability measure on $\SS_r$ which is the pushforward of $\P$ under $\Pi(\cd,0)$. 
We abbreviate $\ul\nu_0^+$ to $\ul\nu$ in this section. 
\begin{lemma}
\label{lem: entropy of RW on Sr}
For every $\ul x\in\supp\ul\nu,$ we have
\[\lim_{n\to + \infty}\frac{1}{n}H(\mu^{*n}*\unif_{\ul x},\cD_{-\lambda n})=-\beta\lambda. \]
\end{lemma}
Recall that $\unif_{\ul x}$ denotes the uniform probability measure on $\ul x \in \SS_r$.
\begin{proof}
By Proposition \ref{prop: convergence in probability}, for every $\ve>0$, and every $n$ large enough depending on $\ve$, we have 
	\[\P(\Sigma') \geq 1-\ve\]
where 
\[\Sigma'=\lb{\omega \in \Sigma :\underline d(f_{\sigma^{-n}\omega}^n \ul x,\Pi(\omega,0)) < 2^{(\lambda+\ve)n}}.\]

From this we can construct a Borel probability measure $\theta_n$ on $\SS^1 \times \SS^1$ whose marginal distributions are $\mu^{*n}*\unif_{\ul x}$ and $\nu$ and moreover
\begin{equation}
\label{eqn: property of theta}
\theta_n \lb{(x,y):d(x,y)\leq 2^{( \lambda+ \ve)n}} \geq 1 - \ve.
\end{equation}
Indeed, write $\ul x = \{x_0,\dotsc,x_{r-1}\}$ and for each $\omega \in \Sigma$, label elements of $\Pi(\omega,0)$ as $y_0(\omega), \dotsc, y_{r-1}(\omega)$ (in a measurable way). 
Define
\[\theta_n \defeq \int_\Sigma \frac{1}{r} \sum_{i=0}^{r-1}\delta_{(f_{\sigma^{-n}\omega}^nx_i,y_i(\omega))} \dd \P(\omega),\]
so that $\theta$ is a probability measure on $\SS^1 \times \SS^1$.
Its projection to the first coordinate is $\mu^{*n}*\unif_{\ul x}$ and to the second coordinate is $\nu$.
At this stage, we still have the freedom to choose the labelling $\Pi(\omega,0) = \{y_0(\omega), \dotsc, y_{r-1}(\omega) \}$.
For $\omega \in \Sigma'$, there is a unique labelling such that for every $i \in[r]$,  $d(f_{\sigma^{-n}\omega}^nx_i,y_i(\omega)) \leq 2^{(\lambda+\ve)n}$, because $\Pi(\omega,0)$ is $c$-separated for some $c = c(\mu) > 0$ by Lemma~\ref{lem: uniform bound of Pi} and $d(f_{\sigma^{-n}\omega}^n \ul x,\Pi(\omega,0)) < 2^{(\lambda+\ve)n} < c/2$.
This shows \eqref{eqn: property of theta}.

We have the following elementary fact on entropies. 

\begin{lemma}
\label{lem: distance and entropy}
Let $\theta$ be a Borel probability measure on $\SS^1\times\SS^1$ such that 
\[\theta \lb{(x,y):d(x,y)\leq 2^{-(1-\ve)n}} \geq 1 - \ve\]
for some $n \in \NN$ and $\ve>0.$ Let $\eta$ and $\zeta$ be the marginal distributions of $\theta$ on the two coordinates. Then
\[|H(\eta,\cD_n)-H(\zeta,\cD_n)|\ll  \ve n+1. \]
\end{lemma}

Applying this lemma, we obtain
\[|H(\mu^{*n}*\unif_{\ul x},\cD_{-\lambda n})-H(\nu,\cD_{-\lambda n})| \ll \ve (-\lambda) n+ 1.\]
Since $\ve > 0$ is arbitrarily small and $\lim_{n \to +\infty}\frac{1}{n} H(\nu,\cD_n)\to \beta$, the desired convergence follows.
\end{proof}

\begin{lemma}\label{lem: generate similar elements}
Assume that $-\lambda\beta< \hRW(\mu).$ 
Then there are constants $c,C>0$, a closed interval $J'\subset \SS^1$ centered at a point of $\supp\nu$ and infinitely many integers $n \in \NN$, together with positive numbers $\ve_n\to 0$ and subsets $T_n\sbs T_\mu,$ satisfying 
\begin{enumerate}
\item $\# T_n\geq 2^{cn},$
\item for every $f\in T_n,$ $f'|_{J'}\in[2^{(\lambda-\ve_n)n},2^{(\lambda+\ve_n)n}]$ and $\vk_\alpha(f,J')\leq C,$ 
\item the intervals $f(J')$, $f\in T_n$, lie in a common interval of length at most $2^{(\lambda+\ve_n)n}$.
\end{enumerate} 
\end{lemma}
\begin{proof}
By the assumption, we can take $c>0$ such that $-\lambda\beta+10\ell c<h_{\mr{RW}}(\mu),$ where $\ell=\max\lb{\#\supp\mu,-\lambda}.$ 

By concavity~\eqref{eqn: entropy concave}, for all $n \in \NN$,
\[ \frac{1}{r}\sum_{x \in \ul x} H(\mu^{*n}*\delta_{x},\cD_{-\lambda n}) \leq H(\mu^{*n}*u_{\ul x},\cD_{-\lambda n}).\]
Thus, by Lemma~\ref{lem: entropy of RW on Sr}, there is $x \in \ul x$ such that
\begin{equation}
\label{eqn: entropy of RW on S}
\liminf_{n\to + \infty}\frac{1}{n}H(\mu^{*n}*\delta_x,\cD_{-\lambda n}) \leq -\lambda \beta.
\end{equation}
Recall Lemma~\ref{lem: not atomic}. 
Let $\ve'>0$ be small so that 
\[\forall i\in\Zmodd,\quad\nu_i^-(B(x,3\ve'))\leq \frac{c}{dr}. \]
Let
\[J' = [x - \ve', x + \ve'].\]

Let $\Sigma_\ve\sbs\Sigma$, $C=C(\ve)$ and $U(\omega)=U(\omega,\ve)$ be given by Proposition \ref{prop: good words} applied to $\ve=\min\lb{c,\ve'}$.
By the uniform convergence and a uniform distortion bound in $\Sigma_\ve,$ there exists positive constants $\ve_n\to 0$ such that for every $y\in U(\omega)\cap W^s(\omega,0),$
\[|\frac{1}{n}\log (f_\omega^n)'(y)-\lambda|\leq\ve_n.\] 
Since $\ul x \in \supp \ul \nu$, for $\P$-almost every $\omega \in \Sigma$, we have $\ul x \sbs W^s(\omega,0)$.
In particular, $J' \cap W^s(\omega,0) \ne \vn$. 
Thus, for $\omega\in \Sigma_\ve,$ if $J' \not\sbs U(\omega),$ then the $2\ve'$-neighborhood of $J'$ must contain a point in $\Xi(\omega).$ 
Thus,
\[\P\lb{\omega \in \Sigma_\ve: J' \not\sbs U(\omega)} \leq \P\lb{\omega: \Xi(\omega) \cap B(x,3\ve') \ne\vn } \leq r\sum_{j \in \Zmodd} \nu_j^-( B(x,3\ve'))\leq c.\]
It follows that the set
\[\Sigma'=\lb{\omega\in\Sigma_\ve :  J'  \sbs U(\omega) \cap W^s(\omega,0)},\]
satisfies $\P(\Sigma')\geq 1-2c.$ 

For every $n\in\NN,$ define $\Sigma_n'\defeq\lb{f_\omega^n:\omega\in\Sigma'},$ then $\mu^{*n}(\Sigma_n')\geq 1-2c.$ 
Denote $\tau_n=1-\mu^{*n}(\Sigma_n')\leq 2c$ and let
\[\mu_n=\frac{1}{1-\tau_n}\mu^{*n}|_{\Sigma_n'}.\]
By the concavity and almost convexity of the entropy~\eqref{eqn: entropy concave}, we have
\[|H(\mu_n)-H(\mu^{*n})|\leq 1+\tau_n(\#\supp\mu)n\leq 1+2\ell cn\]
and
\[|H(\mu_n*\delta_x,\cD_{-\lambda n})-H(\mu^{*n}*\delta_x,\cD_{-\lambda n})|\leq 1+ \tau_n(-\lambda)n\leq 1+2\ell cn.\]
Remembering $\frac{1}{n}H(\mu^{*n})\to h_{\mr{RW}}(\mu)$, the choice of $c$ and \eqref{eqn: entropy of RW on S}, we have 
\[H(\mu_n)\geq H(\mu_n*\delta_x,\cD_{-\lambda n})+cn\]
for infinitely many $n$.

Define $e_x:\Diff^{1+\alpha}_+(\SS^1)\to \SS^1$ to be the map $f\mapsto f(x),$ then $e_x^{-1}\cD_{-\lambda n}$ is a finite partition of $\Diff^{1+\alpha}_+(\SS^1).$ We have
\[H(\mu_n|e_x^{-1}\cD_{-\lambda n})=H(\mu_n)-H(\mu_n,e_x^{-1}\cD_{-\lambda n})=H(\mu_n)-H(\mu_n*\delta_x,\cD_{-\lambda n})\geq cn.\]
This implies the existence of $I \in \cD_{-\lambda n}$ such that
the set $T_n = \lb{f \in \Sigma_n' : f(x) \in I}$ has cardinality $\# T_n\geq 2^{cn}.$ 
By the definition of $\Sigma_n'$, for any $f \in T_n$,
\begin{enumerate}
	\item $\vk_\alpha(f,J')\leq C,$ and
	\item $\forall y\in J'$, $\log f'(y)\in [(\lambda-\ve_n)n,(\lambda+\ve_n)n]$.
\end{enumerate}
Moreover, since $f(x) \in I$, we conclude that $f(J')$ is contained in the  $2^{(\lambda+\ve_n)n}$-neighborhood of $I$. 
After replacing $\ve_n$ by $\ve_n+2/n,$ we obtain that the intervals $f(J')$ lie in a common interval of length at most $2^{(\lambda+\ve_n)n}$ for $f\in T_n.$
\end{proof}

\begin{lemma}\label{lem: approximate identity}
Let $c,C,J',\ve_n, T_n$ be as in Lemma~\ref{lem: generate similar elements}.
Let $J=\frac{1}{2}J'$ be the closed interval of the same center as $J'$ and half the length.
Then there exists $g_n\ne f_n\in T_n$ such that $f_n(J)\sbs g_n(J')$ and the sequence $g_n^{-1}f_n$ converges to the identity on $J$ in the $C^1$-topology.
\end{lemma}
\begin{proof}
Define $k=k(n)$ for large $n$ to be the largest integer such that
\[4k+1<\frac{cn}{\lceil2\ve_nn+\sqrt n\rceil+\log(\lceil2\ve_n n^2\rceil+1)}. \]
Since $\ve_n \to 0$, $k$ is positive for large enough $n$ and $k(n)\to +\infty$ as $n \to +\infty$. 
Let $y_0,y_1,\cdots,y_{4k} \in J'$ be equally spaced points in $J',$ with $J' = [y_0, y_{4k}]$.
Then $d(y_j,y_{j+1})\leq 1/(4k).$ 
For every $f\in T_n,$ each $f(y_i)$ takes value in a fixed interval of length $2^{(\lambda+\ve_n)n}$ and $\log f'(y_i)$ takes value in $[(\lambda-\ve_n)n,(\lambda+\ve_n)n].$ 
Arrange vectors 
\[(f(y_0),\cdots,f(y_{4k}),\log f'(y_0),\cdots,\log f'(y_{4k}) ), f\in T_n\]
into boxes of dimension $2^{(\lambda-\ve_n)n-\sqrt{n}} \times \dotsm \times 2^{(\lambda-\ve_n)n-\sqrt{n}}\times \frac{1}{n} \times \dotsm \times \frac{1}{n}$.
The choice of $k$ guarantees that
\[\sb{2^{\lceil 2\ve_n n+\sqrt n\rceil}}^{4k+1} \sb{\lceil 2\ve_n n^2\rceil +1}^{4k+1}<2^{cn}.\]
Thus, by the pigeonhole principle there exist $f \ne g\in T_n$ such that
\[\forall i = 0, \dotsc, 4k,\quad d(f(y_i),g(y_i)) \leq 2^{(\lambda-\ve_n)n-\sqrt{n}} \quad \text{and}\quad |\log f'(y_i)-\log g'(y_i)| \leq \frac{1}{n}.\]
	
The endpoints of $J$ are exactly $y_k$ and $y_{3k}.$ 
Note that, for $n$ large enough,
\[f(y_k) \geq g(y_k) - 2^{(\lambda-\ve_n)n-\sqrt{n}} \geq g(y_0) + \frac{|J'|}{4}\min_{y \in J'} g'(y) - 2^{(\lambda-\ve_n)n-\sqrt{n}}  \geq g(y_0).\]
Similarly $f(y_{3k}) \leq g(y_{4k})$.
It follows that \[f(J)\sbs g(J').\]

In what follows, we will estimate the $C^1$ distance between $g^{-1}f$ and $\Id$ on $J$.
First, since $g^{-1}:f(J)\to J'$, $g^{-1}$ is $2^{(-\lambda+\ve_n)n}$-Lipschitz on $f(J).$ 
Hence for every $i = 0, \dotsc, 4k$, 
\[d(g^{-1}f(y_i),y_i)\leq2^{(-\lambda+\ve_n)n}d(f(y_i),g(y_i))\leq 2^{-\sqrt{n}}.\]
Then, using $\vk_\alpha(g,J')\leq C$,
\[|\log (g^{-1}f)'(y_i)|\leq\abs{\log\frac{f'(y_i)}{g'(y_i)}}+\abs{\log\frac{g'(y_i)}{g'(g^{-1}fy_i)}}\leq \frac{1}{n}+C d(g^{-1}f(y_i),y_i)^\alpha\leq \frac{1}{n}+C2^{-\alpha\sqrt n}.\]
	
More generally, let $y\in J$ be arbitrary. 
There is $i \in \{k,\dotsc, 3k\}$ such that $y\in [y_{i},y_{i+1}].$
Then $g^{-1}f(y)\in [g^{-1}f(y_i),g^{-1}f(y_{i+1})],$ both points being contained in the $2^{-\sqrt{n}}$ neighborhood of $[y_i,y_{i+1}]$.
Hence
\[d(y,g^{-1}f(y))\leq \frac{1}{4k}+2^{-\sqrt{n}+1}.\]
Recall that we have both the estimates $\vk_\alpha(f,J')\leqslant C$ and $\vk_\alpha(g,J')\leqslant C.$
The logarithm of derivatives at $y$ can be bounded by comparing with $\log (g^{-1}f)'(y_i)$:
	\[\abs{\log\frac{(g^{-1}f)'(y)}{(g^{-1}f)'(y_i)} }\leq\abs{\log\frac{f'(y)}{f'(y_i)}}+\abs{\log\frac{g'(g^{-1}f y_i)}{g'(g^{-1}fy)}}\leq C\frac{1}{(4k)^\alpha}+C\sb{\frac{1}{2k}+2^{-\sqrt n+1}}^\alpha .\]
	
Now, for each $n$, construct $(f_n,g_n)$ in this way. 
Since $k\to +\infty$ as $n\to +\infty$, we conclude that $g_n^{-1}f_n$ tends to the identity on $J$ in the $C^1$-topology.
\end{proof}

\begin{proof}[Proof of Theorem \ref{thm: stationary measure dimension}: the case $I=\SS^1$]	
Let $\lambda=\lambda(\mu,\nu).$
By the property of Furstenberg entropy and Theorem \ref{thm: exact dimensionality}, we have
	\[\beta=\dim\nu=-\frac{h_{\mr F}(\mu,\nu)}{\lambda}= -\frac{h_{\mr F}(\mu^{*n},\nu)}{n\lambda}\leq -\frac{H(\mu^{*n})}{n\lambda}\]
	for every positive integer $n.$ 
Letting $n \to +\infty$, we obtain that $\beta\leq -h_{\mr{RW}}/\lambda.$ 
Assume that (1) is violated, that is $-\lambda\beta<h_{\mr{RW}}(\mu)$.
Recall that $J$ is centered at a point of $\supp\nu$.
Combining Lemma~\ref{lem: generate similar elements} and Lemma~\ref{lem: approximate identity} yields the required  subinterval $J$ and two sequences $(f_n),(g_n)\sbs T_\mu.$
\end{proof}
\begin{proof}[Proof of Theorem \ref{thm: dim formula C1+}]
	Let $\Lambda$ be the unique minimal set of $\pair{\supp\mu}.$
	We first show that $\supp\nu\sbs\Lambda.$
	Assume without loss of generality that $\Lambda\ne\SS^1.$
	Let $\cJ$ be the collection of connected components of $\SS^1\sm\Lambda.$
	
	If $\supp\nu\not\sbs\Lambda$ then $\nu(J)>0$ for some $J\in \cJ.$
	Let $c\defeq\max\lb{\nu(J):J\in\cJ}$ and consider the finite subset $\cJ_c\defeq\lb{J\in\cJ:\nu(J)=c}\sbs\cJ.$
	Since $\nu$ is $\mu$-stationary, $g(J)\in\cJ_c$ for every $g\in\supp\mu$ and $J\in\cJ_c.$
	Set $F\defeq\bigcup_{J\in \cJ_c}\partial J.$
	Then $F$ is a finite subset of $\SS^1$ which is invariant under $\supp\mu$ and hence $\pair{\supp\mu},$ which violates the assumption.
	
	Since $\supp\nu\subset\Lambda$, local discreteness excludes alternative \eqref{alt: locally nondiscrete} for every ergodic stationary measure. 
	Combining Theorem \ref{thm: exact dimensionality} and the case $I=\SS^1$ of Theorem \ref{thm: stationary measure dimension} gives the claimed dimension formula and entropy identity.
\end{proof}

\subsection{Dimension formula on an interval}\label{subsec: dim formula on intervals}

The purpose of this subsection is to prove the remaining case of Theorem \ref{thm: stationary measure dimension}.
That is, $I$ is a closed subinterval of $\SS^1$ which is preserved by the support of $\mu$ and there is no common $T_\mu$-invariant probability measure on $I.$
In this case, $T_\mu$ is allowed to preserve a common invariant measure outside $I.$

Let $I\sbs \SS^1$ be a closed subinterval.
Let $C^{1+\alpha}_+(I,I)$ be the semigroup of orientation preserving $C^{1+\alpha}$-maps from $I$ to $I$.
For a diffeomorphism $f\in \Diff_+^{1+\alpha}(\SS^1)$ which preserves the interval $I,$ let $\cR_I(f)\in C_+^{1+\alpha}(I,I)$ be the restriction $f|_I.$
Then $\cR_I: \Diff_+^{1+\alpha}(\SS^1)\to C_+^{1+\alpha}(I,I)$ is well-defined on the sub-semigroup of diffeomorphisms that preserve $I.$
For a probability measure $\mu$ on $\Diff_+^{1+\alpha}(\SS^1)$ whose support preserves $I,$ we usually consider the probability measure $(\cR_I)_*\mu$, the induced probability measure on $C_+^{1+\alpha}(I,I).$

Note that $(\cR_I)_*\mu$ induces a random walk on $I.$ 
We can also consider the random walk entropy of $(\cR_I)_*\mu$ in the semigroup $C_+^{1+\alpha}(I,I).$
We denote $h_{\mr{RW},I}(\mu)\defeq h_{\mr{RW}}((\cR_I)_*\mu).$
It follows immediately from the definition that $ h_{\mr{RW},I}(\mu)\leqslant h_{\mr{RW}}(\mu)$ and the equality holds if $\cR_I: T_\mu\to C_+^{1+\alpha}(I,I)$ is injective. 
Note that if $\cR_I$ is not injective on $T_\mu$ then we can find $f\ne g\in T_\mu$ which are identical on $I.$
Letting $f_n\equiv f$ and $g_n\equiv g$ gives item (2) of Theorem \ref{thm: stationary measure dimension}.
Hence the case that $I$ is a closed interval of Theorem \ref{thm: stationary measure dimension} follows directly from the following proposition.

\begin{proposition}\label{prop: dim formula circle interval}
	Let $\mu$ be a finitely supported probability measure on $\Diff_+^{1+\alpha}(\SS^1)$ such that $\supp\mu$ does not preserve any common probability measure on $\SS^1$. 
	Let $I\sbs\SS^1$ be a closed interval which is preserved by $T_{\mu}.$ 
	Let $\nu$ be an ergodic $\mu$-stationary measure with $\supp\nu\sbs I.$ 
	Then $\nu$ is exact dimensional and
	\begin{enumerate}
		\item either $|\lambda(\mu,\nu)|\geq h_{\mr{RW},I}(\mu)$ and $\dim\nu=\frac{h_{\mr{RW},I}(\mu)}{|\lambda(\mu,\nu)|},$
		\item or there exists a closed interval $J\sbs I$ with $J^\circ\cap\supp\nu\ne\vn$ and two sequences of diffeomorphisms $\{g_n\},\{f_n\}\sbs T_\mu$ with $g_n|_I\ne f_n|_{I}$ and $f_n(J)\sbs g_n(I),$ such that $g_n^{-1}f_n$ tends to $\Id$ on $J$ in the $C^1$-topology.
	\end{enumerate}
\end{proposition}

\begin{lemma}
	Let $\mu$ be a finitely supported probability measure on $\Diff^{1+\alpha}_+(\SS^1).$
	Let $I$ be $\SS^1$ or a closed subinterval of $\SS^1$ which is preserved by $T_\mu.$
	Assume that there is no common $T_\mu$-invariant probability measure on $I.$
	Then there exists a finitely supported probability measure $\wt\mu$ on $\Diff_+^{1+\alpha}(\SS^1)$ such that the following hold:
	\begin{enumerate}[\indent(1)]
		\item $T_{\wt\mu}$ also preserves $I$ and $(\cR_I)_*\wt\mu=(\cR_I)_*\mu.$
		\item There is no common $T_{\wt\mu}$-invariant probability measure on $\SS^1.$
	\end{enumerate}
\end{lemma}

\begin{proof}
		Take a closed interval $J$ on $\SS^1$ disjoint from $I.$ 
		Take $f_1,f_2\in\supp\mu$ such that $f_1$ does not preserve the left endpoint $x_-$ of $I $ and $f_2$ does not preserve the right endpoint $x_+$ of $I$ (it is possible that $f_1=f_2$). 
		Write $\SS^1=I\sqcup U_1\sqcup J\sqcup U_2,$ where $x_-$ is an endpoint of $U_1$ and $x_+$ is an endpoint of $U_2.$ 
		Take $\wt f_1\in\Diff_+^{1+\alpha}(\SS^1)$ such that $\wt f_1(U_1\cup J)\sbs I$ and $\wt f_1|_I=f_1,$ then take $\wt f_2\in\Diff_+^{1+\alpha}(\SS^1)$ such that $\wt f_2(J\cup U_2)\sbs I$ and $\wt f_2|_I=f_2.$ 
		Then the common invariant probability measure of $\wt f_1$ and $\wt f_2$ must be supported on $I.$ 
		For every other element $f\in \supp\mu,$ extend $f$ to $\wt f\in\Diff_+^{1+\alpha}(\SS^1)$ arbitrarily. 
		Write $\mu=\sum_{i=1}^k{p_i}\delta_{f_i}$ and set $\wt\mu=\sum_{i=1}^k{p_i}\delta_{\wt f_i}$.
\end{proof}

\begin{proof}[Proof of Proposition \ref{prop: dim formula circle interval}]
	Let $\wt\mu$ be a finitely supported probability measure given by the previous lemma.
	Then $\nu$ is also an ergodic $\wt\mu$-stationary measure.
	Hence $\nu$ is exact dimensional, with dimension $\beta.$
	Denoting $\lambda=\lambda(\wt\mu,\nu),$ we have
	\[\beta=\dim\nu=-\frac{h_{\mr F}(\wt\mu,\nu)}{\lambda}=-\frac{h_{\mr F}(\wt\mu^{*n},\nu)}{n\lambda}=-\frac{h_{\mr F}((\cR_I)_*\wt\mu^{*n},\nu)}{n\lambda}\leq -\frac{H((\cR_I)_*\wt\mu^{*n})}{n\lambda}.\]
	This implies that $ \beta\leqslant -h_{\mr{RW},I}(\wt\mu)/\lambda .$
	
	If $\beta<-h_{\mr{RW},I}(\wt\mu)/\lambda$, we aim to construct an interval $J$ and two sequences of diffeomorphisms satisfying (2). 
	To achieve this, we need a variant of Lemma \ref{lem: generate similar elements}.
	
	\begin{lemma}\label{lem: generate similar elements 2}
	Let $\wt \mu$ be a finitely supported probability measure on $\Diff_+^{1+\alpha}(\SS^1).$
	Assume that there is no common $T_{\wt\mu}$-invariant probability measure on $\SS^1$ and there is a closed interval $I$ preserved by $T_{\wt\mu}.$
	If $-\lambda\beta< h_{\mr{RW},I}(\wt\mu)$ then there are constants $c,C>0$, a closed interval $J'\subset I$ centered at a point of $\supp\nu$ and infinitely many integers $n \in \NN$ with a sequence of positive numbers $\ve_n\to 0$ and subsets $T_n\sbs T_{\wt\mu},$ satisfying 
	\begin{enumerate}
	\item $\#\cR_I(T_n)\geq 2^{cn},$
	\item for every $f\in T_n,$ $f'|_{J'}\in[2^{(\lambda-\ve_n)n},2^{(\lambda+\ve_n)n}]$ and $\vk_\alpha(f,J')\leq C,$ 
	\item $f(J')$, $f\in T_n$, falls in a common interval of length at most $2^{(\lambda+\ve_n)n}$.
	\end{enumerate} 
	\end{lemma}
	\begin{proof}
	The proof is the same as that of Lemma \ref{lem: generate similar elements} and we give only an outline. 
	The only adaptation needed is considering $(\cR_I)_*\wt\mu^{*n} $ instead of $\mu^{*n}.$
	Then for a point $x\in\supp\nu,$ we obtain
	\[\lim_{n\to\infty} \frac{1}{n}H( (\cR_I)_*\wt\mu^{*n}*\delta_x,\cD_{-\lambda n})=-\beta\lambda <h_{\mr{RW},I}(\wt\mu)=h_{\mr{RW}}((\cR_I)_*\mu)=\lim_{n\to\infty} \frac{1}{n}H( (\cR_I)_*\wt\mu^{*n}).\]
	Then we can use a similar pigeonhole argument to find exponentially many elements in $\supp(\cR_I)_*\mu^{*n}\sbs\cR_I(T_{\wt\mu})$ satisfying (2) and (3).
	\end{proof}
	By our choice of $\wt\mu,$ we have $\cR_I(T_\mu)=\cR_I(T_{\wt\mu}).$
	Combining this lemma with Lemma \ref{lem: approximate identity}, we conclude that 
	\begin{enumerate}
		\item either $h_{\mr{RW},I}(\mu)=h_{\mr{RW},I}(\wt\mu)=h_{\mr F}(\wt\mu,\nu)=h_{\mr F}(\mu,\nu)$ and hence $ \beta= -h_{\mr{RW},I}(\mu)/\lambda$,
		\item or there exist two sequences of maps $(g_n),(f_n)\sbs \cR_I(T_{\wt\mu})=\cR_I(T_\mu)$ with $g_n\ne f_n$ and a subinterval $J\sbs I$ centered at a point of $\supp\nu$ that satisfies $f_n(J)\sbs g_n(I)$ and $g_n^{-1}f_n\to\Id$ on $J$ in the $C^1$-topology, which coincides with (2) in the proposition. \qedhere
	\end{enumerate}
\end{proof}

\section{Approximation by uniformly hyperbolic subsystems}\label{se:7}
\subsection{Hyperbolic elements}\label{subsec: hyperbolic elements}\label{se:7.1}
The construction of hyperbolic fixed points and hyperbolic elements in subgroups and subsemigroups of $\Diff(\SS^1)$ dates back to Sacksteder's Theorem~\cite{Sa65} for $C^2$ pseudo-groups.
Deroin-Kleptsyn-Navas~\cite[Théorème F]{DKN07} showed the existence of hyperbolic elements in subgroups of $\Diff_+^1(\SS^1)$ that do not preserve any probability measure on $\SS^1$.

Here, we prove analogous results for $C^{1+\alpha}$ subsemigroups.
Similar to the proof in~\cite{DKN07}, we use random walks to find hyperbolic elements.
With the description of the random walks given by Theorem~\ref{thm: structure of random walk 1}, the proof is rather straightforward.

Recall that a fixed point $x \in \SS^1$ of a diffeomorphism $f \in \Diff_+(\SS^1)$ is hyperbolic if $f'(x) \neq 1$.
\begin{lemma}
\label{lem: 2dr fixed points}
Let $\mu$, $d$ and $r$ be as in Theorem~\ref{thm: structure of random walk 1}.
Then there is an element in $T_\mu$ having exactly $2dr$ fixed points in $\SS^1$, all of which are hyperbolic.  
\end{lemma}

In fact, the proof of Lemma \ref{lem: 2dr fixed points} allows us to localize the attracting and repelling fixed points, as formalized in the next lemma. We follow the notation in Theorem~\ref{thm: structure of random walk 1}. 
Let $\ul\nu^+ = \Pi_* \P$ and $\ul \nu^- = \Xi_*\P$.
These are probability measures on $\SS_{dr}$, the space of all $dr$-element subsets of $\SS^1$.
For a subset $A \sbs \SS^1$ and $\rho > 0$ write 
\[A^{(\rho)} = \bigcup_{a \in A} B(a,\rho)\]
for the $\rho$-neighborhood of $A$ in $\SS^1$.

\begin{lemma}
\label{lem: A R and f}
Let $A \in \supp \ul\nu^+$ and $R \in \supp \ul\nu^-$.
Assume that $A \cap R = \vn$ and that $A \cup R$ is $3\rho$-separated for some $\rho > 0$.
Then there is $f \in T_\mu$ such that
\begin{equation}
\label{eqn: f sm R to A}
f(\SS^1 \sm R^{(\rho)}) \sbs A^{(\rho)},
\end{equation}
it preserves each connected component of $A^{(\rho)}$, that is,
\begin{equation}
\label{eqn: f A to A id}
\forall a \in A,\quad f(B(a,\rho)) \sbs B(a,\rho),
\end{equation}
and
\begin{equation}
\label{eqn: f contracts on sm R}
f'|_{\SS^1\sm R^{(\rho)}}<1 \quad \text{and} \quad (f^{-1})'|_{\SS^1\sm A^{(\rho)}}<1.
\end{equation}
\end{lemma}

\begin{proof}
Consider
\begin{equation}
\label{eqn: Sigma0 A R}
\Sigma_0 = \lb{\, \omega\in \Sigma : \ul d(\Pi(\omega),A) < \rho/2 \text{ and } \ul d(\Xi(\omega),R) < \rho/2 \,}.
\end{equation}
The two conditions defining $\Sigma_0$ are independent because $\Pi(\omega)$ depends only on $\omega^-$ and $\Xi(\omega)$ depends only on $\omega^+$.
Thus, since $A \in \supp \ul \nu^+$ and $R \in \supp \ul \nu^-$, we have $\P(\Sigma_0) > 0$. By Theorem \ref{thm: structure of random walk 1} (5), elements of $A$ and elements of $R$ arrange alternatively on $\SS^1$.
By Theorem~\ref{thm: supports of stationary measures}, the closed sets $\supp \nu_i$, $i \in [d]$ are disjoint.
Thus we can partition $A = \bigsqcup_{i \in [d]} A_i$ with $A_i = A \cap \supp \nu_i$.
Moreover, if $\rho > 0$ is small enough, then the condition $\ul d(\Pi(\omega),A) < \rho/2$ implies that for all $i \in [d]$, $\ul d(\Pi(\omega,i),A_i) < \rho / 2$.
Thus,
\begin{equation}
\label{eqn: d Pi omega i A i}
\forall \omega \in \Sigma_0,\, \forall i \in [d],\quad \ul d(\Pi(\omega,i),A_i) < \rho / 2.
\end{equation}

By Proposition~\ref{prop: good words} there is a subset $\Sigma' \subset \Sigma$ and constants $c,C > 0$ such that $\P(\Sigma') \geq 1 - \P(\Sigma_0)/2$ and that for all $\omega \in \Sigma'$,
\begin{equation}
\label{eqn: sm Xi contracts}
\forall n \in \NN,\, \forall x \in \SS^1 \sm \Xi(\omega)^{(\rho/2)},\quad  (f^n_\omega)'(x) \leq C 2^{-cn} .
\end{equation}
By Proposition~\ref{prop: good words} applied to $F^{-1}$, similarly, there is a subset $\Sigma'' \subset \Sigma$ such that $\P(\Sigma'') \geq 1 - \P(\Sigma_0)/2$ and that for all $\omega \in \Sigma''$,
\begin{equation}
\label{eqn: sm Pi contracts}
\forall n \in \NN,\, \forall x \in \SS^1 \sm \Pi(\omega)^{(\rho/2)},\quad (f^{-n}_\omega)'(x) \leq  C 2^{-cn}.
\end{equation}
In particular, writing $\Sigma_1 = \Sigma_0 \cap \Sigma'$ and $\Sigma_2 = \Sigma_0 \cap \Sigma''$, we have $\P(\Sigma_1) > 0$ and  $\P(\Sigma_2) > 0$.

Fix $a_0\in A_0$ for the rest of the proof of Lemma \ref{lem: A R and f}.
Let $U_{0,0}$ denote $B(a_0,\rho)$.
By~\eqref{eqn: d Pi omega i A i}, for every $\omega \in \Sigma_0$ there is a unique element in $\Pi(\omega,0) \cap U_{0,0}$, which we will denote by $\Pi(\omega,0,0)$.
Then (recall Theorem~\ref{thm: structure of random walk 1})
\[m^+_0(\Sigma_1 \times U_{0,0}) = \int_{\Sigma_1} u_{\Pi(\omega,0)}(U_{0,0}) \dd \P(\omega) = \frac{1}{r}\P(\Sigma_1) > 0.\]
Similarly, $m^+_0(\Sigma_2 \times U_{0,0}) > 0$. 
By Birkhoff's ergodic theorem, for $m^+_0$-almost every $(\omega,x) \in \Sigma_1 \times U_{0,0}$, there are infinitely many $n \in \NN$ such that $F^n(\omega,x) \in \Sigma_2 \times U_{0,0}$.
In other words, for $\P$-almost every $\omega \in \Sigma_1$, there are infinitely many $n \in \NN$ such that 
\begin{equation}
\label{eqn: f Pi omega 0 0}
\sigma^n\omega \in \Sigma_2\quad \text{and} \quad f^n_\omega \Pi(\omega,0,0) \in U_{0,0}.
\end{equation}
Choose $n$ sufficiently large so that it satisfies \eqref{eqn: f Pi omega 0 0} and $C2^{-cn} < \rho/2$.
We claim that $f^n_\omega$ satisfies the desired properties.

Indeed since $\omega \in \Sigma_0$, we have $\Xi(\omega) \sbs R^{(\rho/2)}$ and hence $\SS^1\sm R^{(\rho)} \sbs \SS^1 \sm \Xi(\omega)^{(\rho/2)}$.
Also, $\ul d(\Pi(\omega), A) < \rho/2$. 
By the separation of $A \cup R$, $\Pi(\omega)\sbs A^{(\rho)} \sbs \SS^1 \sm R^{(\rho)}$.
For a subset $X \sbs \SS^1$, we denote by $\CCs(X)$ the set of its connected components. We call $\ul x \sbs \SS^1$  a  
\textit{representative of $\CCs(X)$} if it contains exactly one point from each component of $X$.
Then $\Pi(\omega)$ is a representative of $\CCs(\SS^1 \sm R^{(\rho)})$.
Thus, by $\omega \in \Sigma'$ and~\eqref{eqn: sm Xi contracts}, for any $\ul x \in \SS_{dr}$ which is a representative of $\CCs(\SS^1 \sm R^{(\rho)})$,
\[\ul d(f^n_\omega \ul x, \Pi(\sigma^n \omega)) = \ul d(f^n_\omega \ul x, f^n_\omega \Pi(\omega)) \leq C 2^{-cn} < \rho/2.\]
But $\sigma^n \omega \in \Sigma_0$, hence $\ul d(\Pi(\sigma^n \omega), A) < \rho/2$.
By the triangle inequality, 
\[\ul d(f^n_\omega \ul x, A) < \rho,\]
showing $f^n_\omega(\SS^1 \sm R^{(\rho)}) \sbs A^{(\rho)}$ and that $(f^n_\omega)_* \colon \CCs( \SS^1 \sm R^{(\rho)} ) \to \CCs(A^{(\rho)})$ is bijective.
It follows that $(f^n_\omega)_* \colon \CCs(A^{(\rho)}) \to \CCs(A^{(\rho)})$ is bijective.
Moreover, it preserves the cyclic order and maps $U_{0,0} \in \CCs(A^{(\rho)})$ to itself.
Hence $(f^n_\omega)_* \colon \CCs(A^{(\rho)}) \to \CCs(A^{(\rho)})$ is the identity map.
Also, from $\omega \in \Sigma'$ and~\eqref{eqn: sm Xi contracts}, we obtain
\[\forall x \in \SS^1 \sm R^{(\rho)},\quad  (f^n_\omega)'(x) \leq C 2^{-cn} < 1.\]
Similarly, from $\sigma^n \omega \in \Sigma_2$ and ~\eqref{eqn: sm Pi contracts}, we obtain $\SS^1 \sm A^{(\rho)} \subset \SS^1 \sm \Pi(\sigma^n \omega)^{(\rho/2)}$ and then
\[\forall x \in \SS^1 \sm A^{(\rho)},\quad (f^{-n}_{\sigma^n \omega})'(x) \leq C 2^{-cn}  < 1.\]
This completes the proof because $(f^n_\omega)^{-1} = f^{-n}_{\sigma^n \omega}$.
\end{proof}

\begin{proof}[Proof of Lemma~\ref{lem: 2dr fixed points}]
Let $h \in T_\mu$ be the map supplied by Lemma~\ref{lem: A R and f} for arbitrary $A \in \supp \ul\nu^+$ and $R \in \supp \ul\nu^-$ with $A \cap R = \vn$ and sufficiently small $\rho > 0$. The existence of such a pair of $A$ and $R$ follows from the continuity of stationary measures, Lemma \ref{lem: not atomic}.

For every $a \in A$, by~\eqref{eqn: f A to A id} and~\eqref{eqn: f contracts on sm R}, $h|_{B(a,\rho)}$ is a contraction on $B(a,\rho)$, hence has a unique attracting fixed point in $\ol{B(a,\rho)}$.
Note that the condition \eqref{eqn: f A to A id} implies that $h^{-1}$ preserves each of the connected components of $R^{(\rho)}$.
Thus similarly, $h$ has a unique repelling fixed point on each of the $dr$ connected components of $R^{(\rho)}$.
Finally, $h$ has no other fixed points because $h(\SS^1 \sm R^{(\rho)}) \sbs A^{(\rho)}$.
\end{proof}

\subsection{Perfect ping-pong pairs}\label{sec: perfect ping-pong pairs}\label{se:7.2}
In this subsection, we will define and construct the perfect ping-pong pairs in every finitely generated sub-semigroup of $\Diff_+^{1+\alpha}(\SS^1)$ that preserves no probability measure on $\SS^1$. 
Roughly speaking, a perfect ping-pong pair is a pair of diffeomorphisms generating several independent ping-pong systems.

Recall the notion of attractors and repellors in Definition \ref{def: attractors and repellors}. A hyperbolic fixed point $x$ of a diffeomorphism $f$ is an attractor (resp. repellor) if and only if $f'(x)<1$ (resp. $f'(x)>1$).

\begin{definition}\label{def: perfect ping-pong pair}
Let $q$ be a positive integer.
A pair $(h_1,h_2)\sbs\Diff_+^1(\SS^1)$ is called a \textit{($q$-)perfect ping-pong pair} if there exist subsets $U_1^+,U_2^+,U_1^-,U_2^-\sbs\SS^1$ satisfying the following conditions:
	\begin{enumerate}[\indent(a)]
		\item Each $h_i$ has exactly $2q$ fixed points, all of which are hyperbolic.
		\item For every attractor $a_1 \in \SS^1$ of $h_1$ there is an attractor $a_2 \in \SS^1$ of $h_2$ such that either $I = [a_1,a_2]$ or $I = [a_2,a_1]$ satisfies $I \cap \Fix(h_i) = \{a_i\}$, for $i \in \{1,2\}$.
        Moreover, the same holds for repellors. 
 		\item The set $U_1^+$ (resp. $U_2^+,U_1^-,U_2^-$) is a union of $q$ disjoint open intervals, each of which is an open neighborhood of an attractor of $h_1$ (resp. $h_2,h_1^{-1},h_2^{-1}$).
		\item The closures of the sets $U_1^+,U_2^+,U_1^-,U_2^-$ are pairwise disjoint.
		\item The collection $\lb{U_i^\pm}$ forms a ping-pong partition for $h_1$ and $h_2$ in the following sense: 
		\[h_1(\SS^1\sm U_1^-)\sbs U_1^+,~h_2(\SS^1\sm U_2^-)\sbs U_2^+,~h_1^{-1}(\SS^1\sm U_1^+)\sbs U_1^-,~h_2^{-1}(\SS^1\sm U_2^+)\sbs U_2^-.\]
		\item The maps $h_i$ and their inverses uniformly contract outside the corresponding cones:
		\[h_1'|_{\SS^1\sm U_1^-}<1,~h_2'|_{\SS^1\sm U_2^-}<1,~(h_1^{-1})'|_{\SS^1\sm U_1^+}<1,~(h_2^{-1})'|_{\SS^1\sm U_2^+}<1.\]
	\end{enumerate}
We call the sets $U_{i}^\pm$ the \textit{ping-pong cones} of the pair $(h_1,h_2).$
\end{definition}
\begin{remark}
Condition (b) is crucial and can be equivalently formulated as follows:
 $\Fix(h_1)$ and $\Fix(h_2)$ are disjoint and we may label the attractors in $\Fix(h_1) \cup \Fix(h_2)$ as $a_1, \dotsc, a_{2q}$ and the repellors as $r_1, \dotsc, r_{2q}$ so that in cyclic order we have
 \[a_1 \prec a_2 \prec r_1 \prec r_2 \prec a_3 \prec a_4 \prec \dotsi \prec a_{2q - 1} \prec a_{2q} \prec r_{2q - 1} \prec r_{2q} \prec a_1.\]
Under this labelling, each interval of the form $[a_{2i-1},a_{2i}]$ is preserved by both $h_1$ and $h_2$, and $(h_1,h_2)$ acts on it with ping-pong dynamics.
Likewise, each interval of the form $[r_{2i-1},r_{2i}]$ is  preserved by both $h_1^{-1}$ and $h_2^{-1}$, and on it $(h_1^{-1}, h_2^{-1})$ exhibits ping-pong dynamics.
\end{remark}

\begin{figure}[!ht]
	\begin{minipage}[c]{7cm}
      \begin{figureexample}\label{eg: perfect pingpong pair}
	Here is an example of a $2$-perfect pingpong pair. The red points represent attractors, and the blue points represent repellors. The black arcs denote the pingpong cones $U_i^\pm$. The arrows inside the circle show the action of $h_1$, and arrows outside the circle represent the action of $h_2$.

    Note that the fixed points of $h_1$ and $h_2$ are not necessarily arranged alternately. Therefore, there is some flexibility in the orders of each pair of attractors or repellors.
	\end{figureexample}
   \end{minipage}%
   \hspace{1cm}
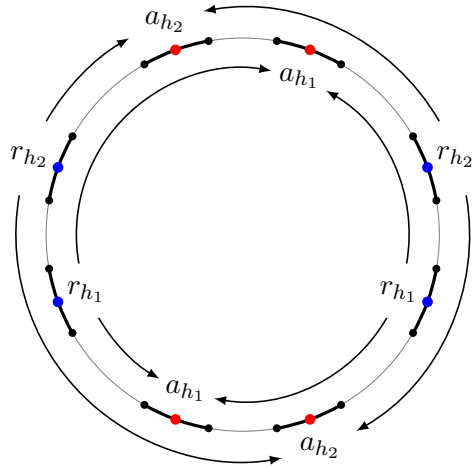
   \begin{minipage}[c]{\textwidth}
       \begin{tikzpicture}
          \def\radius{2.6cm}
		\def\radone{2.2cm}
		\def\radtwo{3cm}
		\draw[gray] (0,0) circle (\radius);
		\draw[-,very thick] (10:\radius) arc[radius=\radius, start angle=10, end angle=30];
		\draw[-,very thick] (60:\radius) arc[radius=\radius, start angle=60, end angle=80];
		\draw[-,very thick] (100:\radius) arc[radius=\radius, start angle=100, end angle=120];
		\draw[-,very thick] (150:\radius) arc[radius=\radius, start angle=150, end angle=170];
		\draw[-,very thick] (190:\radius) arc[radius=\radius, start angle=190, end angle=210];
		\draw[-,very thick] (240:\radius) arc[radius=\radius, start angle=240, end angle=260];
		\draw[-,very thick] (280:\radius) arc[radius=\radius, start angle=280, end angle=300];
		\draw[-,very thick] (330:\radius) arc[radius=\radius, start angle=330, end angle=350];
		
		\fill[red] (0,0) ++(70:\radius) circle[radius=2pt];
		\fill[red] (0,0) ++(110:\radius) circle[radius=2pt];
		\fill[red] (0,0) ++(250:\radius) circle[radius=2pt];
		\fill[red] (0,0) ++(290:\radius) circle[radius=2pt];
		
		\fill[blue] (0,0) ++(20:\radius) circle[radius=2pt];
		\fill[blue] (0,0) ++(-20:\radius) circle[radius=2pt];
		\fill[blue] (0,0) ++(160:\radius) circle[radius=2pt];
		\fill[blue] (0,0) ++(200:\radius) circle[radius=2pt];
		
		\fill (0,0) ++(10:\radius) circle[radius=1.5pt];
		\fill (0,0) ++(30:\radius) circle[radius=1.5pt];
		\fill (0,0) ++(60:\radius) circle[radius=1.5pt];
		\fill (0,0) ++(80:\radius) circle[radius=1.5pt];
		\fill (0,0) ++(100:\radius) circle[radius=1.5pt];
		\fill (0,0) ++(120:\radius) circle[radius=1.5pt];
		\fill (0,0) ++(150:\radius) circle[radius=1.5pt];
		\fill (0,0) ++(170:\radius) circle[radius=1.5pt];
		\fill (0,0) ++(190:\radius) circle[radius=1.5pt];
		\fill (0,0) ++(210:\radius) circle[radius=1.5pt];
		\fill (0,0) ++(240:\radius) circle[radius=1.5pt];
		\fill (0,0) ++(260:\radius) circle[radius=1.5pt];
		\fill (0,0) ++(280:\radius) circle[radius=1.5pt];
		\fill (0,0) ++(300:\radius) circle[radius=1.5pt];
		\fill (0,0) ++(330:\radius) circle[radius=1.5pt];
		\fill (0,0) ++(350:\radius) circle[radius=1.5pt];

		\node at (70:\radone) {$a_{h_1}$};
		\node at (250:\radone) {$a_{h_1}$};
		\node at (200:\radone) {$r_{h_1}$};
		\node at (-20:\radone) {$r_{h_1}$};
		
		\draw[->,>=latex,semithick] (-10:\radone) arc[radius=\radone, start angle=-10, end angle=60];
		\draw[<-,>=latex,semithick] (80:\radone) arc[radius=\radone, start angle=80, end angle=190];
		\draw[->,>=latex,semithick] (210:\radone) arc[radius=\radone, start angle=210, end angle=240];
		\draw[<-,>=latex,semithick] (260:\radone) arc[radius=\radone, start angle=260, end angle=330];
		
		\node at (110:\radtwo) {$a_{h_2}$};
		\node at (290:\radtwo) {$a_{h_2}$};
		\node at (20:\radtwo) {$r_{h_2}$};
		\node at (160:\radtwo) {$r_{h_2}$};
		
		\draw[->,>=latex,semithick] (30:\radtwo) arc[radius=\radtwo, start angle=30, end angle=100];
		\draw[<-,>=latex,semithick] (120:\radtwo) arc[radius=\radtwo, start angle=120, end angle=150];
		\draw[->,>=latex,semithick] (170:\radtwo) arc[radius=\radtwo, start angle=170, end angle=280];
		\draw[<-,>=latex,semithick] (300:\radtwo) arc[radius=\radtwo, start angle=300, end angle=370];

       \end{tikzpicture}
   \end{minipage}
   
\end{figure}

\begin{figure}[!ht]
\begin{minipage}[c]{8cm}
       \begin{tikzpicture}
          \def\radius{2.6cm}
        \def\radthree{2.9cm}
		\def\radone{2.2cm}
		\def\radtwo{1.8cm}
		\draw[gray] (0,0) circle (\radius);
		\draw[-,very thick] (10:\radius) arc[radius=\radius, start angle=10, end angle=30];
		\draw[-,very thick] (60:\radius) arc[radius=\radius, start angle=60, end angle=80];
		\draw[-,very thick] (100:\radius) arc[radius=\radius, start angle=100, end angle=120];
		\draw[-,very thick] (150:\radius) arc[radius=\radius, start angle=150, end angle=170];
		\draw[-,very thick] (190:\radius) arc[radius=\radius, start angle=190, end angle=210];
		\draw[-,very thick] (240:\radius) arc[radius=\radius, start angle=240, end angle=260];
		\draw[-,very thick] (280:\radius) arc[radius=\radius, start angle=280, end angle=300];
		\draw[-,very thick] (330:\radius) arc[radius=\radius, start angle=330, end angle=350];
		
		\fill[blue] (0,0) ++(70:\radius) circle[radius=2pt];
		\fill[red] (0,0) ++(110:\radius) circle[radius=2pt];
		\fill[blue] (0,0) ++(250:\radius) circle[radius=2pt];
		\fill[red] (0,0) ++(290:\radius) circle[radius=2pt];
		
		\fill[blue] (0,0) ++(20:\radius) circle[radius=2pt];
		\fill[red] (0,0) ++(-20:\radius) circle[radius=2pt];
		\fill[blue] (0,0) ++(160:\radius) circle[radius=2pt];
		\fill[red] (0,0) ++(200:\radius) circle[radius=2pt];
		
		\fill (0,0) ++(10:\radius) circle[radius=1.5pt];
		\fill (0,0) ++(30:\radius) circle[radius=1.5pt];
		\fill (0,0) ++(60:\radius) circle[radius=1.5pt];
		\fill (0,0) ++(80:\radius) circle[radius=1.5pt];
		\fill (0,0) ++(100:\radius) circle[radius=1.5pt];
		\fill (0,0) ++(120:\radius) circle[radius=1.5pt];
		\fill (0,0) ++(150:\radius) circle[radius=1.5pt];
		\fill (0,0) ++(170:\radius) circle[radius=1.5pt];
		\fill (0,0) ++(190:\radius) circle[radius=1.5pt];
		\fill (0,0) ++(210:\radius) circle[radius=1.5pt];
		\fill (0,0) ++(240:\radius) circle[radius=1.5pt];
		\fill (0,0) ++(260:\radius) circle[radius=1.5pt];
		\fill (0,0) ++(280:\radius) circle[radius=1.5pt];
		\fill (0,0) ++(300:\radius) circle[radius=1.5pt];
		\fill (0,0) ++(330:\radius) circle[radius=1.5pt];
		\fill (0,0) ++(350:\radius) circle[radius=1.5pt];

		\node at (70:\radone) {$r_{h_1}$};
		\node at (250:\radone) {$r_{h_1}$};
		\node at (200:\radone) {$a_{h_1}$};
		\node at (-20:\radone) {$a_{h_1}$};
		
		\draw[<-,>=latex,semithick] (-10:\radone) arc[radius=\radone, start angle=-10, end angle=60];
		\draw[->,>=latex,semithick] (80:\radone) arc[radius=\radone, start angle=80, end angle=190];
		\draw[<-,>=latex,semithick] (210:\radone) arc[radius=\radone, start angle=210, end angle=240];
		\draw[->,>=latex,semithick] (260:\radone) arc[radius=\radone, start angle=260, end angle=330];
		
		\node at (110:\radtwo) {$a_{h_2}$};
		\node at (290:\radtwo) {$a_{h_2}$};
		\node at (20:\radtwo) {$r_{h_2}$};
		\node at (160:\radtwo) {$r_{h_2}$};

		\draw[->,>=latex,semithick] (30:\radtwo) arc[radius=\radtwo, start angle=30, end angle=100];
		\draw[<-,>=latex,semithick] (120:\radtwo) arc[radius=\radtwo, start angle=120, end angle=150];
		\draw[->,>=latex,semithick] (170:\radtwo) arc[radius=\radtwo, start angle=170, end angle=280];
		\draw[<-,>=latex,semithick] (300:\radtwo) arc[radius=\radtwo, start angle=300, end angle=370];
		
		\draw[->,>=latex,semithick] (80:\radthree) arc[radius=\radthree, start angle=80, end angle=285];
		
		\draw[<-,>=latex,semithick] (295:\radthree) arc[radius=\radthree, start angle=295, end angle=420];

       \end{tikzpicture}
   \end{minipage}%
	\begin{minipage}[c]{\textwidth-9cm}
      \begin{figureexample}
	Here is an example of a pair of uniformly hyperbolic elements $(h_1,h_2)$ with a ``wrong" order on circle. We can obtain this example by replacing $h_1$ by $h_1^{-1}$ in previous example.
	
	Let $h_3=h_2h_1$ (the arrows outside the circle), then $h_3$ only has one repellor and one attractor.
	
	This example shows that for a wrong arrangement, it is possible to find a further subsystem with fewer fixed points.
	
	\end{figureexample}
   \end{minipage}
\end{figure}

\vspace{-0.2cm}
\begin{figure}[!ht]
	\begin{minipage}[c]{7cm}
      \begin{figureexample}
	Here is another example of a pair of uniformly hyperbolic elements $(h_1,h_2)$ with a ``wrong" order. 
	
	Let $h_3=h_2h_1$ (the arrows outside circle). Then, $h_3$ does not have any fixed points, which means that the semigroup generated by $(h_1,h_2)$ is not uniformly hyperbolic in some sense.

    These examples show that for a wrong order of fixed points, the dynamics of the semigroup generated by $(h_1,h_2)$ may not be as clear.
	
	\end{figureexample}
   \end{minipage}%
   \hspace{1cm}
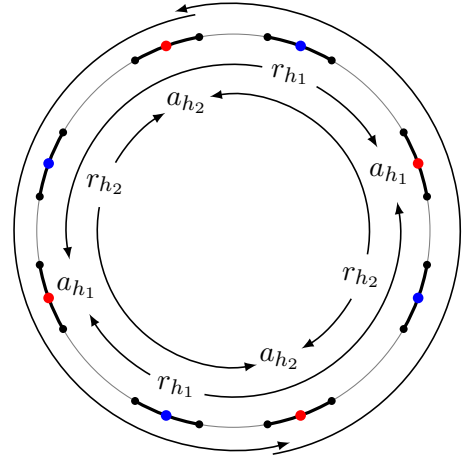
   \begin{minipage}[c]{\textwidth}
       \begin{tikzpicture}
          \def\radius{2.6cm}
        \def\radthree{2.9cm}
        \def\radfour{3cm}
		\def\radone{2.2cm}
		\def\radtwo{1.8cm}
		\draw[gray] (0,0) circle (\radius);
		\draw[-,very thick] (10:\radius) arc[radius=\radius, start angle=10, end angle=30];
		\draw[-,very thick] (60:\radius) arc[radius=\radius, start angle=60, end angle=80];
		\draw[-,very thick] (100:\radius) arc[radius=\radius, start angle=100, end angle=120];
		\draw[-,very thick] (150:\radius) arc[radius=\radius, start angle=150, end angle=170];
		\draw[-,very thick] (190:\radius) arc[radius=\radius, start angle=190, end angle=210];
		\draw[-,very thick] (240:\radius) arc[radius=\radius, start angle=240, end angle=260];
		\draw[-,very thick] (280:\radius) arc[radius=\radius, start angle=280, end angle=300];
		\draw[-,very thick] (330:\radius) arc[radius=\radius, start angle=330, end angle=350];
		
		\fill[blue] (0,0) ++(70:\radius) circle[radius=2pt];
		\fill[red] (0,0) ++(110:\radius) circle[radius=2pt];
		\fill[blue] (0,0) ++(250:\radius) circle[radius=2pt];
		\fill[red] (0,0) ++(290:\radius) circle[radius=2pt];
		
		\fill[red] (0,0) ++(20:\radius) circle[radius=2pt];
		\fill[blue] (0,0) ++(-20:\radius) circle[radius=2pt];
		\fill[blue] (0,0) ++(160:\radius) circle[radius=2pt];
		\fill[red] (0,0) ++(200:\radius) circle[radius=2pt];
		
		\fill (0,0) ++(10:\radius) circle[radius=1.5pt];
		\fill (0,0) ++(30:\radius) circle[radius=1.5pt];
		\fill (0,0) ++(60:\radius) circle[radius=1.5pt];
		\fill (0,0) ++(80:\radius) circle[radius=1.5pt];
		\fill (0,0) ++(100:\radius) circle[radius=1.5pt];
		\fill (0,0) ++(120:\radius) circle[radius=1.5pt];
		\fill (0,0) ++(150:\radius) circle[radius=1.5pt];
		\fill (0,0) ++(170:\radius) circle[radius=1.5pt];
		\fill (0,0) ++(190:\radius) circle[radius=1.5pt];
		\fill (0,0) ++(210:\radius) circle[radius=1.5pt];
		\fill (0,0) ++(240:\radius) circle[radius=1.5pt];
		\fill (0,0) ++(260:\radius) circle[radius=1.5pt];
		\fill (0,0) ++(280:\radius) circle[radius=1.5pt];
		\fill (0,0) ++(300:\radius) circle[radius=1.5pt];
		\fill (0,0) ++(330:\radius) circle[radius=1.5pt];
		\fill (0,0) ++(350:\radius) circle[radius=1.5pt];

		\node at (70:\radone) {$r_{h_1}$};
		\node at (250:\radone) {$r_{h_1}$};
		\node at (200:\radone) {$a_{h_1}$};
		\node at (20:\radone) {$a_{h_1}$};
		
		\draw[<-,>=latex,semithick] (30:\radone) arc[radius=\radone, start angle=30, end angle=60];
		\draw[->,>=latex,semithick] (80:\radone) arc[radius=\radone, start angle=80, end angle=190];
		\draw[<-,>=latex,semithick] (210:\radone) arc[radius=\radone, start angle=210, end angle=240];
		\draw[->,>=latex,semithick] (260:\radone) arc[radius=\radone, start angle=260, end angle=370];
		
		\node at (110:\radtwo) {$a_{h_2}$};
		\node at (290:\radtwo) {$a_{h_2}$};
		\node at (-20:\radtwo) {$r_{h_2}$};
		\node at (160:\radtwo) {$r_{h_2}$};

		\draw[->,>=latex,semithick] (-10:\radtwo) arc[radius=\radtwo, start angle=-10, end angle=100];
		\draw[<-,>=latex,semithick] (120:\radtwo) arc[radius=\radtwo, start angle=120, end angle=150];
		\draw[->,>=latex,semithick] (170:\radtwo) arc[radius=\radtwo, start angle=170, end angle=280];
		\draw[<-,>=latex,semithick] (300:\radtwo) arc[radius=\radtwo, start angle=300, end angle=330];

		\draw[->,>=latex,semithick] (100:\radthree) arc[radius=\radthree, start angle=100, end angle=285];
		\draw[->,>=latex,semithick] (280:\radfour) arc[radius=\radfour, start angle=280, end angle=465];

       \end{tikzpicture}
   \end{minipage}
\end{figure}

\newpage
The main goal of this subsection is to prove the following result on the existence of perfect ping-pong pairs.
\begin{proposition}\label{prop: generate ping-pong pair}
Let $T\sbs \Diff_+^{1+\alpha}(\SS^1)$ be a finitely generated semigroup without invariant probability measures on $\SS^1.$ 
Let $\Delta$ be a minimal set of $T.$ 
Given a sufficiently short open interval $I\sbs \SS^1$ that intersects $\Delta$, there exists a perfect ping-pong pair $(h_1,h_2)\sbs T$, with ping-pong cones $(U^\pm_i)$, such that both $U_1^+$ and $U_2^+$ have a unique connected component contained in $I$, and these components have nonempty intersections with $\Delta$.
\end{proposition}
\begin{proof}
Let $\mu$ be the uniform measure on a finite generating set of $T$ and consider the random walk it induces on $\SS^1$.
Then $T=T_\mu$ and $\Delta$ supports an ergodic $\mu$-stationary measure $\nu.$
Without loss of generality, we may assume that $|I|$ is small enough such that for every $i\in [d],$ $\nu_i^-(I)<\frac{1}{2d}.$ 
Since $I$ contains $x\in\Delta=\supp\nu,$ there exists $A_1,A_2\in\supp\ul\nu^+$ with $A_1\cap A_2=\vn$ and both $A_1\cap I,A_2\cap I$ are nonempty. 
Then we can choose $R_1,R_2\in \supp\ul\nu^-$ such that $A_1,A_2,R_1,R_2$ are pairwise disjoint and $R_1\cap I=R_2\cap I=\vn. $

Take $\ve>0$ sufficiently small so that
\begin{enumerate}
    \item $A_1\cup A_2\cup R_1\cup R_2$ is $3\ve$-separated and
    \item both $A_1^{(\ve)}$ and $A_2^{(\ve)}$ have a connected component contained in $I.$
\end{enumerate}

Let $U_i^+=A_i^{(\ve)}$ and $U_i^-=R_i^{(\ve)}$ for $i=1,2.$ 
Lemma~\ref{lem: A R and f} applied twice gives rise to $h_1,h_2\in T_\mu$ such that for every $i=1,2,$
\begin{enumerate}
    \item $h_i(\SS^1\sm U_i^-)\sbs U_i^+,$
    \item $h_i'|_{\SS^1\sm U_i^-}<1$ and $(h_i^{-1})'|_{\SS^1\sm U_i^+}<1,$
    \item $h_i$ preserves each connected component of $ U_i^+.$
\end{enumerate}

It remains to verify Condition~(b) of Definition~\ref{def: perfect ping-pong pair}.
Note that the attractors and repellors of $h_1$ and $h_2$ have the same cyclic order as $\CCs(U_1^+)$, $\CCs(U_2^+)$, $\CCs(U_1^-)$ and $\CCs(U_2^-)$, and hence as $A_1$, $A_2$, $R_1$ and $R_2$.
Thus it suffices to show that for every $a_1 \in A_1$ there is $a_2 \in A_2$ such that between $a_1$ and $a_2$ there is no other element of $A_1\cup A_2 \cup R_1 \cup R_2$.
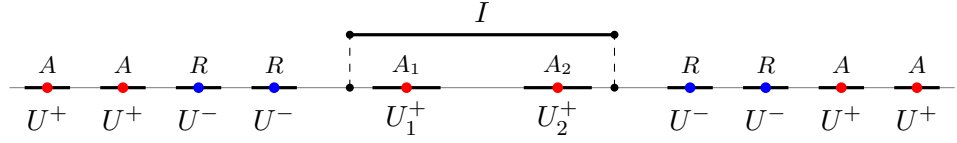
\begin{figure}[!ht]
\centering
	\begin{tikzpicture}
		\draw[-,gray] (0,0) to (12.5,0);
		\draw[-,dashed] (4.5,0) to (4.5,0.7);
		\draw[-,dashed] (8,0) to (8,0.7);
		\draw[-,very thick] (4.8,0) to (5.7,0);
		\draw[-,very thick] (6.8,0) to (7.7,0);
		\draw[-,very thick] (4.5,0.7) to (8,0.7);
		
		\draw[-,very thick] (0.2,0) to (0.8,0);
		\draw[-,very thick] (1.2,0) to (1.8,0);
		\draw[-,very thick] (10.7,0) to (11.3,0);
		\draw[-,very thick] (11.7,0) to (12.3,0);
		
		\draw[-,very thick] (2.2,0) to (2.8,0);
		\draw[-,very thick] (3.2,0) to (3.8,0);
		\draw[-,very thick] (8.7,0) to (9.3,0);
		\draw[-,very thick] (9.7,0) to (10.3,0);
		
		\fill (4.5,0) circle[radius=1.5pt];
		\fill (8,0) circle[radius=1.5pt];
		\fill (4.5,0.7) circle[radius=1.5pt];
		\fill (8,0.7) circle[radius=1.5pt];
		\fill[red] (5.25,0) circle[radius=2pt];
		\fill[red] (7.25,0) circle[radius=2pt];
		
		\fill[red] (0.5,0) circle[radius=2pt];
		\fill[red] (1.5,0) circle[radius=2pt];
		\fill[red] (11,0) circle[radius=2pt];
		\fill[red] (12,0) circle[radius=2pt];
		
		\fill[blue] (2.5,0) circle[radius=2pt];
		\fill[blue] (3.5,0) circle[radius=2pt];
		\fill[blue] (9,0) circle[radius=2pt];
		\fill[blue] (10,0) circle[radius=2pt];

		\node at (6.25,1) {$I$};
		\node at (5.25,0.3) {\footnotesize $A_1$};
		\node at (7.25,0.3) {\footnotesize $A_2$};
		\node at (5.25,-0.4) {$U_1^+$};
		\node at (7.25,-0.4) {$U_2^+$};
		
		\node at (2.5,0.3) {\footnotesize $R$};
		\node at (3.5,0.3) {\footnotesize $R$};
		\node at (2.5,-0.4) {$U^-$};
		\node at (3.5,-0.4) {$U^-$};
		\node at (9,0.3) {\footnotesize $R$};
		\node at (10,0.3) {\footnotesize $R$};
		\node at (9,-0.4) {$U^-$};
		\node at (10,-0.4) {$U^-$};
		
		\node at (0.5,0.3) {\footnotesize $A$};
		\node at (1.5,0.3) {\footnotesize $A$};
		\node at (0.5,-0.4) {$U^+$};
		\node at (1.5,-0.4) {$U^+$};
		\node at (11,0.3) {\footnotesize $A$};
		\node at (12,0.3) {\footnotesize $A$};
		\node at (11,-0.4) {$U^+$};
		\node at (12,-0.4) {$U^+$};
	\end{tikzpicture}
\stepcounter{theorem}
\caption{Arrangement of attractors and repellors}
\end{figure}

Indeed, by Theorem \ref{thm: structure of random walk 1}(5) and the construction \eqref{eqn: Sigma0 A R}, for every $i,j\in\lb{1,2},$ the elements of $A_i$ and those of $R_j$ alternate in cyclic order on $\SS^1.$ 
By construction, the interval $I$ contains a point $a_1 \in A_1$ and a point $a_2 \in A_2$ and no other elements of $A_1\cup A_2 \cup R_1 \cup R_2$.
Let $a$ be the next element of $A_1 \cup A_2$ right to $I$. Let $i \in \{1,2\}$ such that $a \in A_i$, then for each $j = 1,2$, $R_j \cap [a_i,a]$ contains exactly one point.
This shows that necessarily, the next two elements of $A_1\cup A_2 \cup R_1 \cup R_2$ right to $a_1$ and $a_2$ must be one from $R_1$ and one from $R_2$.
Thus, by a simple induction, the claim is verified.

Recalling that $R_1\cap I=R_2\cap I=\vn,$ the arrangement of points in $A_1\cup A_2\cup R_1\cup R_2$ implies that each $U_i^+$ has a unique connected component that intersects $I$ and is contained in $I.$
It remains to prove the final assertion of the proposition.
Since $x\in I$, the iterates $h_1^nx$ of $x$ converge to the unique attracting fixed point of $h_1$ in $U_1^+\cap I$ as $n\to+\infty.$ 
Since $\Delta$ is a $T$-invariant closed set and $x\in\Delta,$ this fixed point is contained in $\Delta.$
Hence this connected component intersects $\Delta.$
\end{proof}

\subsection{Proof of Theorems \ref{thm: ACW approximation} and \ref{thm: separating approximation}}\label{sec: ACW approximation}\label{se:7.3}
We first prove Theorem \ref{thm: ACW approximation}. 
Theorem \ref{thm: separating approximation} will follow as a byproduct of the proof.
Our proof relies on the fact that the approximation of the entropy (condition (2)) and the approximation of the dimension (condition (4)) can be obtained independently, and the two approximations can be combined. 

The proof has three steps. 
First, we construct many uniformly hyperbolic elements with a common family of cones (Proposition \ref{prop: choose uniformly hyperbolic elements}). 
Second, a covering argument produces subfamilies whose images are disjoint and whose cardinalities reflect the dimension of the stationary measure (Proposition \ref{prop: approximate the dimension}). 
Third, we combine the entropy and dimension subfamilies while preserving uniform hyperbolicity (Proposition \ref{prop: approximation of the whole space}). 
The resulting iterated function systems satisfy the open set condition, which gives the required lower bounds for the dimensions of their minimal sets.

Throughout this subsection, as in Theorem \ref{thm: ACW approximation}, $\mu$ is a finitely supported probability measure on $\Diff_+^{1+\alpha}(\SS^1)$ without a common invariant probability measure. 
We follow the notation of Theorem~\ref{thm: structure of random walk 1} and Theorem~\ref{thm: structure of random walk 2}.

\subsubsection*{Step I. Finding uniformly hyperbolic elements}
Recall that $\ul\nu^+=\Pi_*\P$ and $\ul\nu^-=\Xi_*\P.$ 
Fix $A\in\supp\ul\nu^+$ and $R\in\supp\ul\nu^-$ with $A\cap R=\vn$. 
Take $\rho>0$ small enough such that $A\cup R$ is $3\rho$-separated. 

As in the proof of Lemma~\ref{lem: A R and f} we partition $A$ into
\[A=A_0\sqcup A_1 \sqcup \dotsb \sqcup A_{d-1}\]
where $A_i=A\cap \supp\nu_i^+.$ 
Set $U^+:=A^{(\rho)}, U^-:=R^{(\rho)}$, and for each $i \in [d]$, $U_i:= A_i^{(\rho)}$.
Each $U_i$ has exactly $r$ connected components. 
Denote the components by $U_{i,j}$, $j\in[r]$, in cyclic order. 
Without loss of generality we can let $U_{0,0}$ be the same as that in the proof of Lemma~\ref{lem: A R and f}. 
For every $\omega\in\Sigma$, if $\Pi(\omega,0)\cap U_{0,0}$ is not empty, then it contains at most one point which we denote by $\Pi(\omega,0,0)$. 
This notation is consistent with the one used in the proof of Lemma \ref{lem: A R and f}.

For $n \in \NN$ and $\ve >0$, let $\cA_{n,\ve}$ be the set of $f \in \cS^{*n}$ such that
\begin{equation}
\label{eqn: def of A 1}
	\forall i \in [d],\,\forall j \in [r],\quad \ol{f(U_{i,j})} \sbs U_{i,j}
\end{equation}
and
\begin{equation}
\label{eqn: def of A 2}
	\forall i \in [d],\, \forall x \in U_i,\quad  \frac{1}{n} \log f'(x)  \in [\lambda_i^+ -\ve, \lambda_i^+ +\ve].
\end{equation}

\begin{lemma}\label{lem: effective approximation}
There exist $\Sigma_1, \Sigma_2 \sbs \Sigma$ with $\P(\Sigma_1)>0$ and $\P(\Sigma_2) > 0$, such that $\Pi(\omega)\sbs U^+$ for every $\omega\in\Sigma_1\cup\Sigma_2.$
Moreover, there exists a positive integer $n_1$ such that for all $n \geq n_1$ and all $\omega \in \Sigma_1$, if $F^n(\omega, \Pi(\omega,0,0)) \in \Sigma_2 \times U_{0,0}$ then $f^n_\omega \in \cA_{n,\ve}$. 
\end{lemma}
\begin{proof}
Lemma \ref{lem: effective approximation} was proven implicitly in Lemma~\ref{lem: A R and f}. We only need to track the exponents more carefully in~\eqref{eqn: sm Xi contracts}.
More precisely, let $\Sigma_0$ be the set given by \eqref{eqn: Sigma0 A R}. 
We choose $\ve' = \min\lb{\rho/2, \P(\Sigma_0)/2}$ and apply Proposition~\ref{prop: good words} to obtain the set $\Sigma_{\ve'}$ of uniform good words. 
Then we define $\Sigma_1 = \Sigma_0 \cap \Sigma_{\ve'}$ and $\Sigma_2 = \Sigma_0$. 
It can be shown that these sets satisfy the desired properties.
\end{proof}

\begin{proposition}\label{prop: choose uniformly hyperbolic elements}
For every $\ve>0,$ there exist $c>0$ and infinitely many $n>0$ satisfying $\mu^{*n}( \cA_{n,\ve}) \geq c$.
\end{proposition}

\begin{proof}
	Let $\Sigma_1,\Sigma_2$ be the sets given by the previous lemma. Let
	\[\wt\Sigma_i=\lb{(\omega,\Pi(\omega,0,0)):\omega\in\Sigma_i}\sbs \Sigma\times\SS^1,\quad\forall i=1,2.\]
	Then $m_0^+(\wt\Sigma_i)=\frac{1}{r}\P(\Sigma_i)>0$ for $i=1,2.$ Recall that $P:\Sigma\times\SS^1\to\Sigma$ is the natural projection. By the previous lemma, for every $n$ sufficiently large and an element
	\[\omega\in P(\wt\Sigma_1\cap F^{-n}\wt\Sigma_2)\sbs P(\wt\Sigma_1\cap F^{-n}(\Sigma_2\times U_{0,0})),\]
	we have $f_\omega^n\in\cA_{n,\ve}.$ Since $P_*m_0^+=\P,$ we have
	\[\mu^{*n}(\cA_{n,\ve})\geqslant \P(P(\wt\Sigma_1\cap F^{-n}\wt\Sigma_2 ))\geqslant m_0^+(\wt\Sigma_1\cap F^{-n}\wt\Sigma_2).\]
	Then the conclusion follows from the fact that $m_0^+(\widetilde{\Sigma}_1 \cap F^{-n}\widetilde{\Sigma}_2)>0$ for infinitely many $n>0$ since $m_0^+$ is an ergodic $F$-invariant measure.
\end{proof}

\subsubsection*{Step II. Approximating the dimension}

\begin{proposition}\label{prop: approximate the dimension}
For every $\ve > 0$ and every $(i,j) \in [d] \times [r]$, there is a sequence of subsets of $\cA_{n,\ve}$, say $(\cB_{n})_{n\in\NN}$, such that
\begin{equation*}
\limsup_{n \to +\infty} \frac{1}{n} \log \#\cB_{n} \geq (|\lambda_i^+|-\ve) \dimH \nu_i^+
\end{equation*}
and for each $n \in \NN$, the intervals  $f(U_{i,j})$, $f\in\cB_{n}$ are pairwise disjoint.
\end{proposition}

\begin{proof}
For $n \in \NN$, let $\cE_n = \lb{\, f(U_{i,j}) : f \in \cA_{n,\ve} \,}$.
By the definition of $\cA_{n,\ve}$, for all $I \in \cE_n$, $|I| \leq 2^{(\lambda_i^+ + \ve)n}$.
Let $\cB_{n}$ be a maximal subset of $\cA_{n,\ve}$ such that the intervals $f(U_{i,j})$, $f \in \cB_{n}$ are pairwise disjoint.
We need the following covering lemma for general fractals.
\begin{lemma}
\label{lem: cover and dim}
Let $C > 0$ and $\lambda > 0$ be parameters. 
For $n \in \NN$, let $\cE_n$ be a collection of intervals in $\SS^1$ of length at most $C2^{-\lambda n}$.
Let $E = \limsup_{n \to +\infty} \bigcup_{J \in \cE_n} J$.
Let $\wt \cE_n \sbs \cE_n$ be a maximal subcollection consisting of pairwise disjoint intervals.
Then
\[\limsup_{n \to +\infty} \frac{1}{n} \log \#\wt\cE_n \geq \lambda \dimH E.\]
\end{lemma}
\begin{proof}

For every $J \in \cE_n \sm \wt\cE_n$, by the maximality of $\wt\cE_n$, there is $I \in \wt\cE_n$ such that $J\cap I \ne \vn$. Then $J \sbs \wt I$ where $\wt I$ denotes the interval centered at the midpoint of $I$ and having length $C 2^{-\lambda n + 2}$.
It follows that
\[\bigcup_{J\in\cE_n } J \sbs \bigcup_{I\in\wt\cE_n} \wt I.\]

Given $\rho > 0$, for any $N \geq \lambda^{-1}( |\log \rho| + \log C + 2)$, we have
\[E \sbs \bigcup_{n \geq N} \bigcup_{I\in\wt\cE_n} \wt I,\]
which is a cover of $E$ by intervals of length at most $\rho$.
Thus for any $s > 0$ and $\rho > 0$, 
\begin{equation*}
H_{\rho}^s(E)
\leq \sum_{n \geq N} \sum_{I\in\wt\cE_n} |\wt I|^s\\
\ll_{C,s} \sum_{n \geq N} 2^{-s \lambda n} \#\wt\cE_n.
\end{equation*}
The right-hand side is the tail of a convergent series whenever $- s \lambda + \beta < 0$ where $\beta = \limsup_{n\to +\infty} \frac{1}{n} \log \#\wt \cE_n$.
In this case, $H^s(E) = 0$. 
Hence $\dimH E \leq \frac{\beta}{\lambda}$.
\end{proof}

Applying Lemma~\ref{lem: cover and dim} to the subcollections $\wt\cE_n = \lb{\, f(U_{i,j}) : f \in \cB_{n} \,}$, we find 
\[\limsup_{n\to\infty}\frac{1}{n}\log\# \cB_{n} \geq (|\lambda_i^+| - \ve)\dimH E,\]
where $E = \limsup_{n \to +\infty} \bigcup_{I \in \cE_n} I$, that is, the set of points which belong to $\bigcup_{I \in \cE_n} I$ for infinitely many $n$.
Note that $E$ is a Borel set.
To conclude, it remains to show that $\nu_i^+(E) > 0$. 

Let $\Sigma_0$ be as in~\eqref{eqn: Sigma0 A R}.
For $\omega \in \Sigma_0$, for each $(i,j) \in [d] \times [r]$, there is a unique element in 
$\Pi(\omega) \cap U_{i,j} = \Pi(\omega,i) \cap U_{i,j}$. We denote it by $\Pi(\omega,i,j)$.
Let $\Sigma_1, \Sigma_2 \sbs \Sigma_0$ be as in Lemma~\ref{lem: effective approximation}. 
Let 
\[\Sigma_2' = \lb{\, \omega\in\Sigma_2 : \text{ there exists infinitely many } n\in \NN,\, F^{-n}(\omega,\Pi(\omega,0,0)) \in \Sigma_1 \times U_{0,0} \,}.\]
By Theorem~\ref{thm: structure of random walk 1}, $m_0^+$ is ergodic for $F^{-1}$; hence we have $\P(\Sigma_2') = \P(\Sigma_2)> 0$.

We claim that for all $\omega \in \Sigma_2'$, $\Pi(\omega,i,j) \in E$.
Indeed, for $\omega \in \Sigma_2'$ there are infinitely many $n \in \NN$ such that $F^{-n}(\omega,\Pi(\omega,0,0)) \in \Sigma_1 \times U_{0,0}$.
Then, $\sigma^{-n}\omega \in \Sigma_1$ and $f_\omega^{-n} \Pi(\omega,0,0)  \in \Pi(\sigma^{-n}\omega,0) \cap U_{0,0}$, hence $f_\omega^{-n} \Pi(\omega,0,0) = \Pi(\sigma^{-n}\omega,0,0)$
and 
$F^n\bigl(\sigma^{-n}\omega,\Pi(\sigma^{-n}\omega,0,0)\bigr) \in \Sigma_2 \times U_{0,0}$.
By Lemma~\ref{lem: effective approximation}, $f^n_{\sigma^{-n}\omega} \in \cA_{n,\ve}$.
Then, $\Pi(\omega,i,j) = f^n_{\sigma^{-n}\omega} \Pi(\sigma^{-n}\omega,i,j) \in f^n_{\sigma^{-n}\omega}(U_{i,j}) \in \cE_n$. This proves the claim.

Then by Theorem~\ref{thm: structure of random walk 1}(2),
	\[\nu_i^+(E)\geq\frac{1}{r} \P\lb{\,\omega: \Pi(\omega,i)\cap E\ne\vn \,} \geq\frac{1}{r}\P(\Sigma_2')>0.\]
Hence $\dimH E \geq \dimH \nu_i^+$, which completes the proof of the proposition.
\end{proof}

\subsubsection*{Step III. Completion of the proof}
Proposition \ref{prop: choose uniformly hyperbolic elements} allows us to approximate the original system by a uniformly hyperbolic system with a ``large set'' in the base space $\Sigma.$ 
Proposition \ref{prop: approximate the dimension} says that we can approximate the original system by a uniformly hyperbolic system with a large dimension along the fiber $\SS^1$. 
It remains to put together these approximations.

\begin{proposition}\label{prop: approximation of the whole space}
For every $\ve>0,$ there exist a constant $c'>0$ and infinitely many positive integers $N$ such that 
\begin{enumerate}
	\item $\mu^{*N}(\cA_{N,\ve})\geq c'.$
	\item For every $(i,j) \in [d] \times [r]$, there are at least $2^{\dim\nu_i^+ (|\lambda_i^+| - 2\ve)  N}$ elements $f\in \cA_{N,\ve}$ such that $f(U_{i,j})$ are pairwise disjoint.
\end{enumerate}
\end{proposition}
\begin{proof}
Observe that for any $n, N \in \NN$, we have $\cA_{n,\ve}^{*N} \sbs \cA_{nN,\ve}$. 
Recall each $\nu_i^+$ is exact dimensional by Theorem \ref{thm: exact dimensionality} and hence $\dim\nu_i^+=\dimH\nu_i^+.$ 
For each $(i,j) \in [d] \times [r],$ by Proposition \ref{prop: approximate the dimension}, there is a positive integer $n_{i,j}$ and a set $\cB_{i,j} \sbs \cA_{n_{i,j},\ve}$ of cardinality $\# \cB_{i,j} \geq 2^{\dim \nu_i^+ (|\lambda_i^+| -2\ve) n_{i,j}}$ and such that $f(U_{i,j})$, $f \in \cB_{i,j}$ are pairwise disjoint subintervals of $U_{i,j}$.
Then item (2) holds
whenever $N$ is a common multiple of $n_{i,j}$, $(i,j) \in [d] \times [r]$. Indeed, the product set $\cB_{i,j}^{*N/n_{i,j}}$ is a subset of $\cA_{N,\ve}$. It has cardinality
\[\# \cB_{i,j}^{*N/n_{i,j}} = (\# \cB_{i,j})^{N/n_{i,j}} \geq 2^{\dim \nu_i^+ (|\lambda_i^+| -2\ve) N}\]
and $f(U_{i,j})$, $f \in \cB_{i,j}^{*N/n_{i,j}}$ are pairwise disjoint.

Let $N_0$ be the least common multiple of  $n_{i,j}$, $(i,j) \in [d] \times [r]$.
By Proposition \ref{prop: choose uniformly hyperbolic elements}, there is a constant $c > 0$ such that there are infinitely many $n \in \NN$ such that $\mu^{*n}(\cA_{n,\ve})\geq c.$
Then for such $n$, we have
\[\mu^{*n N_0}(\cA_{n N_0,\ve}) \geq \mu^{*n N_0}(\cA_{n,\ve}^{*N_0}) \geq c^{N_0}.\]
So the required properties hold for $N = n N_0$ with $c' = c^{N_0}$, finishing the proof.
\end{proof}
\begin{proof}[Proof of Theorem \ref{thm: ACW approximation}]
	Property (1) follows from Theorems \ref{thm: structure of random walk 1}, \ref{thm: d,r top inv} and \ref{thm: structure of random walk 2}, where $\lambda_i=\lambda_i^+$ and $\nu_i=\nu_i^+.$
	We will show that there are infinitely many $N$ such that $\cA_{N,\ve}$ is a suitable choice of $\Gamma.$
	
	We consider $\wt\Gamma=\lb{(f_1,\cdots,f_N)\in\cS^N: f_1\cdots f_N\in\cA_{N,\ve}}.$ 
	By the first item in the previous proposition, we have $\mu(\wt\Gamma)=\mu^{*N}(\cA_{N,\ve})\geqslant c'$ for infinitely many positive integers $N.$
	By the Shannon-McMillan-Breiman theorem, there exists a subset of $\cS^N$ with at least $(1-o(1))$ $\mu^N$-measure such that each element in this subset has  $\mu^N$-measure of $2^{-N(H(\mu)+o(1))}.$ 
	Therefore, for $N$ large enough, $\wt\Gamma$ contains at least $2^{N(H(\mu)-\ve)}$ elements. 
	Hence, (2) holds for a sufficiently large $N$.
	
	The property (3) is indeed the definition of $\cA_{N,\ve},$ the conditions \eqref{eqn: def of A 1} and \eqref{eqn: def of A 2}. 
	
	It remains to show (4). 
	Fix a pair of $(i,j)$. 
	Let $T$ be the semigroup generated by $ \cA_{N,\ve}$ and let $T_{i,j}$ be the semigroup generated by $2^{\dim\nu_i(|\lambda_i|-2\ve)N}$ elements $f\in\cA_{N,\ve}$ such that $f(U_{i,j})$ are pairwise disjoint given by Proposition \ref{prop: approximation of the whole space}. 
	By the uniform contraction of $T|_{U_{i,j}}$ and the fact that $T$ strictly preserves $U_{i,j}$, $T$ admits a unique minimal set $K_{i,j}\sbs U_{i,j}.$ 
	
	We consider a random walk on the interval $\ol{U_{i,j}}.$ Let $\mu'$ be the uniform measure supported on these $2^{\dim\nu_i(|\lambda_i|-2\ve)N}$ elements. 
	Then $\mu'$ has a unique stationary $\nu'$ on $\ol{U_{i,j}}.$ 
	Moreover,
	\[\supp\nu'=\bigcap_{n=1}^\infty\bigcup_{f\in(\supp\mu')^{*n}}f(\ol{U_{i,j}})= K_{i,j}.\]
	Note that for every $f\in \supp\mu',$ we have $f'(x)\geqslant 2^{(\lambda_i-\ve)N}$ for every $x\in \ol{U_{i,j}}.$ 
	Hence every interval of diameter $2^{(\lambda_i-\ve)nN}$ intersects at most $2$ intervals of the form $f(\ol{U_{i,j}}),$ $f\in(\supp\mu')^{*n}.$ 
	We obtain
	\[\forall\rho>0,\quad\sup_{x\in\SS^1}\log\nu'(B(x,\rho))\leqslant \frac{\log\#\supp\mu'}{(|\lambda_i|+\ve)N}\log\rho+ \log\#\supp\mu'+10.\]
	This implies that
	\[\dim_{\mr H}K_{i,j}\geqslant\dim_{\mr H}\nu'\geqslant \frac{\log\#\supp\mu'}{(|\lambda_i|+\ve)N}\geqslant \dim\nu_i~\frac{|\lambda_i|-2\ve}{|\lambda_i|+\ve}.\]
	Then the conclusion follows by shrinking $\ve>0$ if necessary.
\end{proof}

\begin{proof}[Proof of Theorem \ref{thm: separating approximation}]
	By Theorem \ref{thm: dim formula C1+}, we have $(\dim \nu)|\lambda|=h_{\mr{RW}}(\mu).$ Then the statement directly follows from Proposition \ref{prop: approximate the dimension}, using the cardinality estimate established there. 
	For (5), the proof is analogous to the proof above.
\end{proof}

\begin{small}

\addcontentsline{toc}{section}{References}
\bibliographystyle{abbrv}
\bibliography{ref.bib}

\vspace{0.3cm}

\noindent Weikun He\\
Institute of Mathematics, Academy of Mathematics and System Science, Chinese Academy of Sciences\\
No.55 Zhongguancun East Road, Haidian District, Beijing, China, 100190\\
heweikun@amss.ac.cn

\vspace{0.3cm}

\noindent Yuxiang Jiao\\
Beijing International Center for Mathematical Research, Peking University\\
No.5 Yiheyuan Road, Haidian District, Beijing, China, 100871\\
ajorda@pku.edu.cn

\vspace{0.3cm}

\noindent Disheng Xu\\
Department of Science, Great Bay University\\
Building A5, Songshanhu international community, Dongguan, Guangdong, China, 523000\\
xudisheng@gbu.edu.cn

\end{small}

\end{document}